\documentclass[12pt,a4paper,onecolumn]{IEEEtran}

\usepackage{ multirow,array,ragged2e}

\usepackage{amsfonts}
\usepackage{subfig}

\usepackage{amsthm}
\usepackage{verbatim}
\usepackage{amsmath}
\usepackage{amscd}
\usepackage[latin2]{inputenc}
\usepackage{t1enc}
\usepackage[mathscr]{eucal}
\usepackage{indentfirst}
\usepackage{graphicx}
\usepackage{graphics}
\usepackage{pict2e}
\usepackage{epic}

\numberwithin{equation}{section}
\usepackage[margin=2.9cm]{geometry}
\usepackage{epstopdf} 
\usepackage{cite}
\usepackage{color}
\usepackage{marginnote}
\usepackage{MnSymbol,wasysym}
\usepackage{setspace}

\newtheorem{theorem}{Theorem}
\newtheorem{definition}[theorem]{Definition}

\newtheorem{remark}[theorem]{Remark}
\newtheorem{lemma}[theorem]{Lemma}

\newtheorem{proposition}[theorem]{Proposition}

\newcommand {\R}{} \def\R{\ensuremath{\mathbb{R}}}
\newcommand {\N}{} \def\N{\ensuremath{\mathbb{N}}}
\newcommand {\Z}{} \def\Z{\ensuremath{\mathbb{Z}}}

\def\FF{\mathcal{F}}
\def\BB{\mathcal{B}}

\def\P{\mathbb{P}}
\def\dist{\mathop{\rm dist}}

\def\LG{\textcolor{black}}

\def\TK{\textcolor{black}}

\newcommand{\be}{\begin{equation}}
\newcommand{\ee}{\end{equation}}

\graphicspath{{../figures/}}

\begin{document}

\title{Random attraction  in the TASEP model\thanks{Research supported in part by research grants from the Israel Science Foundation and the US-Israel Binational Science Foundation.}}

\author{Lars Gr\"une and Thomas Kriecherbauer and Michael Margaliot\thanks{LG and TK are  with the
Mathematical Institute, University of Bayreuth, Germany.
MM is with the  School of Electrical Engineering and the Sagol School of Neuroscience, Tel Aviv
University, Israel. }}

\maketitle 

\onehalfspace 

\begin{abstract} 
The totally asymmetric simple exclusion process (TASEP) is a basic model of statistical mechanics that has found numerous
  applications. We consider the case of TASEP with a finite chain where particles may enter from the left and leave to the right at prescribed rates. This model can be formulated as a Markov process with a finite number of states. Due to the irreducibility of the process it is well-known that the probability distribution on the states is globally attracted to a unique equilibrium distribution. We extend this result to the more detailed level of individual trajectories. To do so we
 formulate TASEP as a random dynamical system. Our main result is that the trajectories from all possible initial conditions contract to each other yielding the existence of a random attractor that consists of a single trajectory almost surely.
This implies that in the long run TASEP ``filters out'' any
perturbation  that changes the state
 of the particles along the chain.  
 \TK{
 In order to prove our main result we first establish that any random dynamical system on a finite state space possesses both a global random
pullback attractor and a global random forward attractor. This observation appears to be missing in the literature. We then provide sufficient and necessary conditions for these attractors to be singletons. Finally, we show that TASEP satisfies one of these conditions.
 }
\end{abstract}

\maketitle

 
\section{Introduction} 

The \emph{totally asymmetric simple exclusion process}~(TASEP) 
is a fundamental dynamical model  from nonequilibrium  statistical mechanics. 
It was first introduced as a 1D lattice model for
the motion of    ribosomes along the mRNA strand during
translation~\cite{MacDonald1968}.
TASEP describes particles stochastically
 hopping      along a one-directional 1D chain of sites, where each site can be either empty or contain a single particle. 
This simple exclusion principle 
  generates an indirect coupling between the particles, as a particle cannot hop to a site that is already occupied by another particle. TASEP is a generic tool that has been used to model and analyze numerous natural and artificial processes including vehicular traffic, the kinetics of molecular motors, and
ribosome flow along the mRNA during translation~\cite{TASEP_motors,TASEP_tutorial_2011,ScCN11}. Although one-dimensional, TASEP   exhibits phase transitions  between low-density, high-density, and maximum-current phases~\cite{Krug1991,solvers_guide}.

The simple exclusion principle allows to model and analyze the evolution of  particle ``traffic jams''. Indeed, if a particle remains in the same site for a long time then other particles will accumulate in the sites ``behind'' the occupied site. Traffic jams in the flow of ``biological machines''
like ribosomes and molecular motors, have important ramifications and are attracting considerable interest
(see e.g.~\cite{Ross5911,SIMMS20191679,CHOWDHURY2005318}).

The \emph{ribosome flow model}~(RFM)~\cite{reuveni} is the dynamic mean-field approximation of~TASEP. The RFM and its variants have been  used extensively to model mRNA translation of both isolated mRNA molecules~\cite{rfm_chap,RFM_EXTEN}, 
and networks of mRNAs~\cite{nani,Raveh2016}, 
as well as other important cellular processes like phosphorelay~\cite{EYAL_RFMD1}.

 For a chain of~$n$ sites the RFM can be written as set of~$n$ ODEs:
 \begin{align*}  
\dot x_1 &= \alpha(1-x_1)-h_1 x_1(1-x_2) ,\\
\dot x_2 &=  h_1 x_1(1-x_2)- h_2 x_2(1-x_3),\\
&\vdots\\
\dot x_n &=  h_{n-1} x_{n-1}(1-x_n)- \beta x_n, 
\end{align*}
where~$x_i(t) \in [0,1]$  
  represents the normalized density of particles at site~$i$ along the chain, with~$x_i(t)=0$ [$x_i(t)=1$] representing that this site is almost surely
  empty [full] at time~$t$.
 The number $h_i>0$ is the transition rate from site~$i$ to site~$i+1$ and $\alpha$, $\beta >0$ denote the rates at which particles enter the chain from the left or exit to the right, respectively. 
	To explain this model,
	consider the equation for~$\dot x_2$. This states that the change in the density at site~$2$
	is the flow from site~$1$ to site~$2$ minus the flow from site~$2$ to site~$3$. The latter is given by~$h_2 x_2(1-x_3)$
	i.e. it is proportional to the transition rate, the density at site~$2$ and the ``free space'' in site~$3$. 
	This is a ``soft'' version of   simple exclusion.  Just like TASEP, the RFM can be used to model and analyze the evolution of traffic jams. 

 The RFM has been analyzed
 using various tools from systems and control theory including the theory of cooperative dynamical systems~\cite{RFM_stability},
contraction theory~\cite{RFM_entrain},  continued fractions,
the spectral theory of
 tridiagonal matrices~\cite{RFM_MAX}, and more. 

Recall that a dynamical system
\be\label{eq:dynf}
\dot x=f(x),
\ee
with~$x(t) \in \R^n$, is called \emph{contractive} if
there exist a vector norm~$|\cdot|:\R^n\to \R_+$ and~$\eta >0$
such that for any two initial conditions~$a,b$ in the state-space, we have 
\[
|x(t ,a)-x(t ,b) | \leq \exp(-\eta t) |a-b| \text{ for all } t\geq 0 .
\]  
Here~$x(t,a)$ denotes
 the solution of~\eqref{eq:dynf} at time~$t$ for~$x(0)=a$. 
 Thus, any two solutions approach one another at an exponential rate and in particular 
the initial conditions are exponentially ``forgotten''.
It was shown in~\cite{RFM_entrain} that the~RFM is  an (almost) contractive system
(see also~\cite{3gen_cont_automatica}).

This statement also holds on a more detailed level when one considers the time-evolution of the probability distribution on the $2^n$ states of~TASEP. This time-evolution is governed by a linear differential equation that is called \emph{Kolmogorov's forward equation} or  \emph{master equation} in the theory of Markov processes. Since~TASEP defines an irreducible Markov process it is well-known that any initial distribution is contracted to a unique equilibrium distribution that only depends on the transition rates $\alpha$,$\beta$,$h_1,\ldots,h_{n-1}$. In the special 
case where all the internal transition rates $h_i$ are equal an extremely useful representation for the equilibrium measure was derived in~\cite{Derrida_etal_1993} using the matrix product ansatz, see also~\cite{solvers_guide},~\cite{Krug16}. 


  Having observed that densities and probability distributions contract to an equilibrium it is natural to ask: 
do the trajectories of~TASEP also contract? Of course, 
as the individual trajectories do not converge to an equilibrium it only makes sense to ask whether there is contraction between different trajectories. 
Note that the mutual contraction of different trajectories is also called \emph{synchronization}, see e.g.,~\cite[Definition 3.1.1]{newman_2018} \TK{or \emph{coalescence},~\cite[Section 5.2]{LevP17}}. In order to investigate such a question one needs to have a model of TASEP that describes the random jumps of particles simultaneously for all $2^n$ states of the system. The usual Markov process associated with~TASEP does not provide such information as any realization $\omega$ of randomness defines only a single trajectory. 
\LG{In the literature on Markov processes, the concept of a grand coupling, see, e.g.,  \cite[Chapter 5]{LevP17}, provides this information. This concept, in turn, needs a so called graphical construction or graphical representation of the Markov process, a construction that goes back to \cite{Harr78} and is given for TASEP, e.g., in \cite[p.~215]{Ligg99}. Another} standard framework for investigating the question of synchronization is given by \emph{random dynamical systems}~(RDS). \LG{An RDS can be seen as a refinement of the grand coupling that also takes into account solutions starting arbitrary far in the past.} To the best of our knowledge such an~RDS for~TASEP is missing in the literature. \LG{In this paper,} we fill this gap for finite chains \LG{based on a graphical representation that \TK{extends the one presented}
in \cite[p.~215]{Ligg99}.}

\LG{An important advantage of the RDS formulation is that it allows for the construction of {\em random attractors}. One of the main results in this paper, Theorem \ref{thm:singleton}, states that for general finite state RDS almost sure synchronization is equivalent to the existence of a global random attractor in both forward and pullback sense that almost surely consists of single trajectories, as well as to other equivalent conditions. As we will show in Lemma~\ref{lemma:fintime}, the RDS for TASEP satisfies one (and thus all) of these conditions. Hence it synchronizes almost surely and exhibits this particular form of a random attractor. 
\TK{
For the proof of Lemma~\ref{lemma:fintime} a rather crude lower bound on the synchronization time suffices. It should be noted that a number of techniques have been developed to obtain much better bounds in the context of analyzing mixing times for simple exclusion processes, including  TASEP, with open boundaries and constant jump rates in the interior, see \cite{Gantert_etal_2020} and 
the  references therein.
}
}
 
We would like to point out that \TK{almost sure synchronization} is not a 
consequence of the fact that the equilibrium of the master equation
 description of~TASEP, i.e., the stationary distribution, is a global attractor in the space of probability distributions. 
Indeed, in Remark~\ref{rem:rw3}~(ii) we present an example of a Markov process for which the stationary distribution is a global attractor \LG{(as we shown at the end of Section \ref{sec:distributions}), but for which the trajectories do not synchronize, because the random attractor does not consist of single trajectories.} \TK{In the converse direction, however, we prove with Theorem~\ref{thm:RDSdistrT} for an arbitrary RDS on a finite state space that almost sure synchronization implies that the induced evolution of the probability distributions is contractive. If, in addition, the RDS is constructed from a Markov process this yields the global attractivity of the stationary distribution which therefore must also be unique.} \LG{Thus, the property we establish for TASEP in this paper is strictly stronger than the existence of a globally attracting stationary distribution. Generally, many of the techniques and properties we apply in this paper can also be used in order to study the behavior of probability distributions of finite state Markov processes, see, e.g., the books by Ligget~\cite{Ligg99} and by Levin, Peres and Wilmer~\cite{LevP17}. However, we are not aware of results that allow to conclude the dynamic properties for TASEP in the RDS framework proved in this paper from properties of its probability distribution.}


\TK{
The relation between the evolution of distributions and synchronization described in the previous paragraph is very useful when studying mixing times for exclusion processes, see e.g.,~\cite[Chapter 23]{LevP17} and the recent work~\cite{Gantert_etal_2020} where a number of detailed results on mixing times for the simple exclusion process with open boundaries are presented. We illustrate some of the results in~\cite{Gantert_etal_2020}
and numerically validate  a conjecture of~\cite{Gantert_etal_2020} by Monte-Carlo simulations in Section~\ref{sec:NumSim}.
}


It is known~\cite{entrain_master} that convergence to a unique probability distribution holds for the TASEP master equation also in the case where the transition rates are time-varying and jointly periodic. The same is true for the convergence of densities for the RFM~\cite{RFM_entrain}. These results provide a framework for studying entrainment to periodic excitations like the 24h solar day or the 
cell division process   at the genetic level. Of course, for both equations the limiting state is not stationary any more but depends periodically on time. In Section~\ref{sec:extensions} we describe briefly how our result may be extended to this more general situation.


Our result for~TASEP that random dynamics leads to a synchronization of trajectories has been observed in a number of quite different settings and goes back at least to the work of Baxendale and Stroock~\cite[Proposition 4.1]{BaxendaleStroock1988}, see also the discussion in~\cite[Section 4.1]{newman_2018}. One such setting that bears some similarity with our case is the time-discrete dynamics induced by the composition of random maps on the circle. For example, Kleptsyn and  Nalskii~\cite{klep2004}
considered a finite set of  orientation-preserving
homeomorphisms $T_k$, $k=1,\dots,n$, of the circle,
and the dynamical system obtained by
 applying
a randomly and independently chosen~$T_i$ at each time step.
They showed that under certain assumptions on the semigroup generated by the~$T_i$s there exists 
a~$C^1$-open set   of random
dynamical systems for which 
the distance between
the iterates of different points tends to zero as the   number of iterations tends to infinity.

A general framework for such \emph{noise-induced synchronization} (or synchronization
by noise) was introduced by Newman~\cite{newman_2018}. He considered the composition of independent and identically distributed random maps or a memoryless stochastic flow on a compact metric  space~$X$. He derived conditions for 
  almost-sure mutual convergence of any given pair of trajectories (i.e., global synchronization), namely,
  synchronization occurs and is ``stable'' if and
only if the following properties hold: 
(i) there is a smallest non-empty
invariant set~$K\subset X$; 
(ii) any two points in~$K$ are capable of being moved closer together;
and (iii) $K$ admits asymptotically stable trajectories.
In our case (i) is satisfied with $K=X$, condition (ii) is verified by Lemma~\ref{lemma:fintime}, and (iii) is trivially satisfied as our system has a finite number of states. Therefore, \LG{almost sure synchronization} also follows from~\cite[Theorem 4.2.1]{newman_2018}. However, as \cite{newman_2018} uses the more detailed version of a filtered~RDS we present a self-contained (and short) \LG{analysis of this property for finite state RDS} without reference to~\cite[Theorem 4.2.1]{newman_2018}. \LG{This allows us to provide the equivalent characterizations of the synchronization property in Theorem \ref{thm:singleton} already mentioned above, which are} not addressed in~\cite{newman_2018}. 

Finally, we would like to mention a set of stochastic differential equations that is not directly related to our results but belongs to the same theme park. These equations consist of a diffusive or contractive deterministic evolution equation that is subjected to stochastic driving. In this realm the question of synchronization has been addressed
for It\^{o} stochastic differential equations with a contractive deterministic part in~\cite{slotine_stoc_cont}. There two copies of the same deterministic system with different initial conditions and driven by distinct and independent Wiener processes were considered, and a bound for the mean square distance between the
 solutions was derived. A celebrated representative of
 stochastically driven diffusive systems   is the Kardar-Parisi-Zhang~(KPZ) equation of statistical physics. It is conceptually interesting that this equation can be derived \TK{from the simple exclusion process under a specific weakly asymmetric scaling, see e.g.~the survey~\cite{Corwin_2012}. In fact, it was the discovery of a connection between~TASEP and random matrix theory in~\cite{Johansson2000} that showed that~TASEP belongs to the~KPZ universality class and which invigorated interest in the~KPZ equation (see e.g.,~\cite{Corwin_2012, Quastel_Spohn2015} for recent reviews on these developments).} 

The remainder of this paper is organized as follows. The next section reviews the notion of an~RDS. Section~\ref{sec:TASEP} carefully formulates TASEP as an RDS. \LG{Section~\ref{sec:attractors} shows that for a  general finite state RDS global forward and pullback random attractors exist and coincide, and that they  consist of single trajectories almost surely if and only if the RDS is synchronizing almost surely. These results are then applied to TASEP in Section~\ref{sec:TASEPattractors}. Section~\ref{sec:distributions} clarifies the relation between almost sure synchronization and the convergence of the probability distributions.}
Our theoretical results are  demonstrated using numerical simulations in Section~\ref{sec:NumSim}.
\TK{In this section we also present Monte-Carlo simulations that are related to some of the results
and to a conjecture posed  
in \cite{Gantert_etal_2020}. Some of the data obtained by these numerical experiments are deferred to the Appendix.}
Extensions of our main results to time-periodic transition rates and 
to the asymmetric simple exclusion process~(ASEP) are described in Section~\ref{sec:extensions}.
The final section summarizes our findings.


\section{Random dynamical systems}

In order to   analyze  random attraction in TASEP, we make use of the framework of  random dynamical 
systems~(RDS). The definition of an RDS goes back at least as far as \cite{ArnC91}, a comprehensive treatment can be found in the monograph \cite{Arno98}. 
In the literature, one can find several variants of this definition, which differ in minor technical details. We use 
the definition  from~\cite{CraK15}:
\begin{definition}
A continuous time {\em random dynamical system} (RDS) $(\theta,\varphi)$ on a topological space $X$, equipped with the corresponding Borel sigma algebra, consists of 
\begin{itemize}

\item an autonomous measurable and measure-preserving dynamical system $\theta=\{\theta_t\}_{t\in\R}$ acting on a probability space $(\Omega,\FF,\P)$, i.e.,
\begin{eqnarray*} && \mbox{\rm (i) \; } \theta_0(\omega) = \omega, \qquad \mbox{\rm (ii) \; } \theta_{s+t}(\omega) = \theta_s\circ\theta_t(\omega)\\
&& \mbox{\rm (iii) \; } (t,\omega) \mapsto \theta_t(\omega) \mbox{ is measurable}
\end{eqnarray*}
for all $\omega\in\Omega$ and all $s,t\in\R$, such that $\theta_t\P = \P$ for every $t\in\R$, where $\theta_t\P(A)=\P(\theta_t A)$ for all $A\in\FF$\\

\item a cocycle mapping $\varphi:\R_0^+\times X\times\Omega \to X$, i.e.,
\begin{enumerate}
\item[(1)] $\varphi(0,x,\omega) = x$ for all $x\in X$ and $\omega\in\Omega$ \hfill (initial condition)
\item[(2)] $\varphi(s+t,x,\omega) = \varphi(s,\varphi(t,x,\omega),\theta_t\omega)$ for all $s,t,\in\R_0^+$, $x\in X$ and $\omega\in\Omega$\\ \phantom{.} \hfill (cocycle property)
\item[(3)] $(t,x,\omega) \to \varphi(t,x,\omega)$ is measurable \hfill (measurability)
\item[(4)] $x\mapsto \varphi(t,x,\omega)$ is continuous for all $(t,\omega)\in\R_0^+\times \Omega$ \hfill (continuity)
\end{enumerate}
\end{itemize}
\label{def:rds}\end{definition}

The interpretation of the times in this definition is as follows: $\varphi(\cdot,x_0,\theta_{t_0}\omega)$ denotes the solution path corresponding to $\omega$ starting at time $t_0$ in state $x_0$ and $\varphi(t,x_0,\theta_{t_0}\omega)$ denotes the state on this path at time $t+t_0$. This means that the first time argument of $\varphi$ indicates the time that elapsed since the initial time $t_0$, rather than the absolute time. \LG{For finite state Markov processes, the concept of an RDS can be seen as a refinement of the grand coupling as described, e.g., in \cite[Chapter 5]{LevP17}. The main difference is that in the RDS we consider solutions \TK{with starting times} on the whole real axis and that the driving system $\theta$ is also defined in backwards time. This property is crucial for the construction of random attractors that we will investigate in Section \ref{sec:attractors}. However, it also means that the construction of an RDS is more involved than the definition of a grand coupling. For TASEP, this construction is carried out in the next section.}

\section{TASEP as a random dynamical system}\label{sec:TASEP}

The \emph{totally asymmetric simple exclusion process}~(TASEP)
 is a Markov process for particles hopping or jumping along a 1D chain. We consider the continuous time version of TASEP here. Moreover, we restrict ourselves to finite lattices with $n\in\N$ sites. Then the Markov process has only a finite number of states. A particle at site $k\in\{1,\ldots,n-1\}$ hops
to site $k+1$ (the next site on the right)
  at a random jump time that is exponentially distributed with rate $h_k$, provided that site $k+1$ is not occupied by another particle. 
	This simple exclusion property generates an indirect coupling between the particles and allows, e.g.,
	to model the formation of traffic jams. Indeed, 
	if a particle ``gets stuck'' for a long time in the same site then other particles accumulate behind it. 
At the left end of the chain  particles enter with a certain entry rate~$\alpha>0$ and at the right end particles leave with a rate~$\beta>0$. We refer to \TK{\cite{Ligg99}}, \cite{ScCN11} and the references therein for more information about this model.

In the following subsections we describe how TASEP can be written as an~RDS.


\subsection*{The dynamics of hopping in TASEP}
In order to define the state of the system we 
associate to each site $k\in\{1,\ldots,n\}$ a
 variable~$s_k$. 
We set $s_{k}=1$ if   site $k$ is occupied by a particle and $s_k=0$ if it is not. Hence, the (finite) state space of TASEP is $X=\{0,1\}^n$. Since the state space is finite, we use the discrete topology and its Borel sigma algebra, i.e.~all subsets of $X$ are open and measurable.

For defining the dynamics of TASEP, we start with formalizing a single hop via a map $f$ and then a sequence of hops via a map $
\tilde\varphi$. The map $\varphi$ needed in Definition~\ref{def:rds} will then be derived from $\tilde\varphi$ in the last subsection of this paragraph, when the stochastic model defining the jump times has been introduced. \LG{This procedure is similar to the construction of the graphical representation for TASEP sketched on \cite[p.~\TK{215}]{Ligg99}, except that here we also define solutions starting arbitrarily far in the past. This requires a much more involved  construction of the 
Poisson process driving the dynamics, which is carried out in the following subsections.}

A single hop can be defined as follows. We are given a state $x = (s_1,\ldots,s_n)\in X$ and an index $k\in\{0,\ldots,n\}$ of the site at which the particle attempts to hop, where $k=0$ represents a particle entering the chain. Then we define 
\[ f(x,k) := \left\{ \begin{array}{ll} 
(1,s_2,\ldots,s_n) & \mbox{ if } k=0 \\
(s_1,\ldots,s_{n-1},0) & \mbox{ if } k=n \\
(s_1,\ldots,s_{k-1},0,1,s_{k+2},\ldots, s_{n}) & \mbox{ if } k\ne 0, k\ne n, s_k=1 \mbox{ and } s_{k+1}=0\\
x & \text{ otherwise.} 
\end{array}\right.\]
Now assume that we have a sequence of jump times\footnote{Actually, ``jump attempt times'' would be the more accurate name, but as it is also more clumsy so we prefer the shorter ``jump times''.} $(t_i(\omega))_{i\in\Z}$ with $t_i(\omega)\in\R$ and $t_{i}(\omega) < t_{i+1}(\omega)$ for all $i\in\Z$ together with indices $k_i(\omega)\in \{0,\ldots,n\}$ indicating at which site a particle attempts to jump at time $t_i(\omega)$. The argument $\omega$ indicates that these (deterministic) sequences are realizations of random sequences $(t_i)$ and $(k_i)$. We will specify below how we generate these random sequences in order to meet the exponential distribution requirement. The transition $\tilde\varphi:\R^+_0\times X\times\R^\Z\times\{0,\ldots,n\}^\Z\to~X$ mapping the initial value $x_0$ at initial time $0$ to the state $\tilde\varphi(t,x_0,(t_i(\omega)),(k_i(\omega)))$ at time $t$, given the jump time and index sequences $(t_i(\omega))_{i\in\Z}$ and $(k_i(\omega))_{i\in\Z}$ is then defined by $\tilde\varphi(t,x_0,(t_i(\omega)),(k_i(\omega))):=x_0$ if $t_i(\omega) \not\in [0, t)$ for all $i \in \Z$, otherwise inductively via 
\begin{equation} x_{p+1} := f(x_p,k_{p+i_0}) \mbox{ for } p=0,\ldots,\Delta i, \qquad \tilde\varphi(t,x_0,(t_i(\omega)),(k_i(\omega))):= x_{\Delta i+1} \label{eq:tildephi}\end{equation}
where $i_0:=\inf\{ i\in\Z\,|\, t_i(\omega) \ge 0\}$, $i_1:=\sup\{ i\in\Z\,|\, t_i(\omega) < t\}$ and $\Delta i: =i_1-i_0$. Here we assume that $(t_i(\omega))$ has no accumulation points in $\R$, which can be done since our specification of $t_i$, below, will ensure that this 
indeed holds. 

\subsection*{Assignment of the jump times}

When trying to write TASEP as an RDS, it is not enough to define the ``hopping dynamics''. In addition, we face the difficulties described in \cite[top of p.~55]{Arno98}: a model that is described in terms of transition probabilities does not define a unique RDS because it only describes the evolution of single or one-point motions. For defining an RDS, however, we need to specify the simultaneous motion of solutions subject to different initial conditions but identical random influence. \LG{This corresponds to the concept of coupling and graphical representation in the Markovian literature, cf.\ the discussion in the introduction. Such a representation is in general not unique. We discuss two possible variants in the remainder of this subsection, from which we use the graphical representation described in \cite[p.~215]{Ligg99} for our analysis in this paper. Another source of non-uniqueness will be discussed at the end of this section.} 

As the random influence in TASEP is entirely determined by the jump times, \LG{the construction of a graphical representation} requires to specify the relation between the elementary events $\omega\in\Omega$ and the random jump times. 
\LG{In TASEP,} the rule specified for the jump times is that once a particle jumped or attempted to jump, the time to the next jump attempt is exponentially distributed. This \LG{can be achieved by assigning a sequence of random jump times $\tau_{p,l}$ (that we may think of as ``random clocks'')} to each particle $p$, such that the increments $w_{p,l}:= \tau_{p,l+1}-\tau_{p,l}$ are independent and exponentially distributed as well as independent of $\tau_{p,l}$ and of all jump times for all the other particles. 

Attaching the random clocks to the particles, however, has the disadvantage that
one needs to keep track at which site the particle is, as the expected rate $h_k$ at which the clock goes off depends on the location $k$ of the particle. 
We therefore \LG{follow \cite[p.~215]{Ligg99} and assign} the random clocks to the sites $k$. That is, we
model the jumps using sequences of jump times $T_{k,j}$ with exponentially distributed and independent increments $W_{k,j} := T_{k,j+1}-T_{k,j}$, such that $W_{k,j}$ is independent of $T_{k,j}$ and of all jump times associated with the other sites. Besides being more convenient for our subsequent analysis, this definition is also quite natural in view of the fact that the jump rates $\alpha$, $\beta$ and $h_k$ in the model are site-dependent and not particle-dependent. 


\LG{We would like to point out that} it does not matter for the transition probabilities whether we attach the random clocks to the particles or to the sites. Attaching the jump times to the sites is still consistent with the requirement that the difference between any two consecutive jump times is exponentially distributed, even if the corresponding particle hops, i.e., when it changes its site. This is due to the memorylessness of the exponential distribution and the independence assumption: if a particle attempts to jump at time $T=T_{k,l}$ but cannot jump, then the next jump time is $T_{k,l+1}$, whose difference to $T=T_{k,j}$ is exponentially distributed, because by independence for each $t\ge 0$ we have
\[ P(T_{k,j+1}-T \ge t \,|\, T_{k,l}=T) = P(W_{k,j} \ge t \,|\, T_{k,j}=T) = P(W_{k,j}\ge t).\]
If the jump at time $T=T_{k,j}$ is successful, then the next jump time for the particle is $T_{k+1,m+1}$, where $m$ is such that $T_{k+1,m}\le T$ and $T_{k+1,m+1}> T$. In this case, we can exploit the memorylessness, which says that for all $t>s\ge 0$ the identity 
\[ P(W_{k+1,m} \ge t + s \,|\, W_{k+1,m} > s) = P(W_{k+1,m} \ge t)\]
holds. Together with the fact that $T_{k+1,m}\le T$ and $T_{k+1,m+1}>T$ is equivalent to $T_{k+1,m}\le T$ and $W_{k+1,m} > s$ for $s:=T-T_{k+1,m}\ge 0$ and independence of $W_{k+1,m}$ and $T_{k+1,m}$ we obtain
\begin{eqnarray*} && P(T_{k+1,m+1} - T \ge t \,|\, T_{k+1,m}\le T, T_{k+1,m+1}>T )\\
&& \;\; = \;\; P(W_{k+1,m} \ge t + s\,|\, T_{k+1,m}\le T, T_{k+1,m+1}>T )\\
&& \;\; = \;\; P(W_{k+1,m} \ge t + s\,|\, T_{k+1,m}\le T, W_{k+1,m} > s )\\
&& \;\; = \;\; P(W_{k+1,m} \ge t + s\,|\, W_{k+1,m} > s )\\
&&\;\; =\;\;  
 P(W_{k+1,m} \ge t).
\end{eqnarray*}
Hence, at time $T$ the next jump time $T_{k+1,m+1}$ is again exponentially distributed.

In the sequel, we thus 
attach the jump time sequences to the sites. 

\subsection*{Generation of the jump times}
We now give a precise stochastic definition of the random jump time sequences~$(T_{k,j})_{j\in\Z}$.
From the requirement that the increments $T_{k,j+1}-T_{k,j}$ are exponentially distributed and stochastically independent of $T_{k,j}$ it follows that the $(T_{k,j})_j$ can be modeled by a Poisson process. Usually it is of no interest  on which underlying probability space this process is constructed and it is enough to know that it exists. However, for an RDS one is required by Definition \ref{def:rds} to provide a measure-preserving dynamical system $\{\theta_t\}_{t \in \mathbb R}$ acting on the probability space. The role of $\theta_t$ for $t\in \mathbb R$ is revealed by the cocycle 
property~(2) in  Definition~\ref{def:rds}: Associated with $\omega \in \Omega$ and $k \in \{0, \ldots, n\}$ is the sequence of jump times $(T_{k,j}(\omega))_j$ corresponding to the site $k$.
Then the sequence $(T_{k,j}(\omega)-t)_j$ should also be a realization of the Poisson process. Therefore, there should exist an $\omega' \in \Omega$ and a $\Delta j \in \Z$ with $T_{k,j}(\omega)-t=T_{k,j-\Delta j}(\omega')$ for all $j$. This~$\omega'$ is then denoted by~$\theta_t(\omega)$. Note that this construction can   work only
if the Poisson process is defined on the whole real line and not just on $[0, \infty)$, which suffices to construct~TASEP and is usually used there. 


We now provide an explicit construction of a generalized Poisson process on the whole real line that allows us to define $\theta_t$ in the sense just described.
For this purpose it is convenient to think of the Poisson process as a point process. We follow the wonderful book of Kingman \cite{Kingman}. From \cite[Sections 1.3 and 2.1]{Kingman} we extract the following definition.


\begin{definition}\label{def:Ppp}
Let $(\Omega,\mathcal F,\P)$ be a probability space, $|\cdot |$ the Lebesgue measure 
on~$\R$, $\BB(\R)$ the set of Borel sets $A\subset\R$, $\lambda>0$, and let~$\R^{\infty}$ denote the set of all countable subsets of~$\R$.
 A homogeneous  Poisson process on~$\R$ with rate~$\lambda$
 is a map~$\Pi : \Omega \to \R^{\infty}$ satisfying the following 
 three conditions:
\begin{itemize}
\item[(i)] The maps $N(A):\Omega \to \N \cup \{ \infty \}$, $\omega \mapsto \#  (\Pi(\omega) \cap A)$ are measurable for all $A \in \BB(\R)$, i.e.~for all $m \in \N \cup \{ \infty \}$ and all $A \in \BB(\R)$ we have that the set $\{ \omega \in \Omega \mid \Pi(\omega) \cap A$ contains exactly~$m$ points$\}$ 
belongs to the sigma algebra $\mathcal F$. 
\item[(ii)] For any pairwise disjoint sets $A_1,\ldots,A_p\in\BB(\R)$, $p\in\N$, the random variables $N(A_1),\ldots,N(A_p)$ are independent.
\item[(iii)] $N(A)$ is ${\rm Pois}(\lambda |A|)$-distributed for all $A\in\BB(\R)$.
\end{itemize}
\end{definition}






Note that for a homogeneous Poisson process $\Pi$ on $\R$ with positive rate the set $\Pi(\omega) \in \R^{\infty}$ has almost surely no accumulation points in $\R$, because all $\omega$ for which $\Pi(\omega)$ has a finite accumulation point are contained in $\bigcup_{L \in \N} \{ \omega \mid N([-L, L]) = \infty \}$ which is a countable union of sets of zero measure due to condition (iii). In Section~2.5 of~\cite{Kingman} an explicit construction for a rather general class of Poisson processes is presented. For our purposes it is more convenient to proceed in a different way that is described in Section 4.1 of~\cite{Kingman}: Since $\Pi_{+} := \Pi \cap \R_{0}^+$ and $\Pi_{-} := \Pi \cap \R^-$ are independent homogeneous Poisson processes with rate $\lambda$ (cf. the Restriction Theorem~\cite[Section 2.2]{Kingman} and condition~(ii)) and $x \mapsto -x$ maps $\Pi_-$ to a homogeneous Poisson process on $\R^+$ with rate $\lambda$ (cf. the Mapping Theorem~\cite[Section 2.3]{Kingman}) the Interval Theorem \cite[Section 4.1]{Kingman}) allows us to construct $\Pi$ using partial sums of two independent iid sequences of exponentially distributed random variables. The observation made above about the almost sure absence of accumulation points allows us to consider divergent series only. We summarize these considerations in a precise way:

\begin{definition}\label{def:Pi}
Let $\lambda$ be positive and denote by $\mathcal W_0^{\lambda}:=(\Omega_0,\sigma_0, \nu_0^{\lambda})$ the probability space on $\Omega_0:=[0, \infty)$ equipped with the corresponding Borel sigma algebra $\sigma_0$ and the measure $\nu_0^{\lambda}(dx):=\lambda e^{-\lambda x} dx$. For $i \in \Z \setminus \{ 0 \}$ we denote by $\mathcal W_i^{\lambda}$ identical probability spaces that are obtained from $\mathcal W_0^{\lambda}$ by restricting $\Omega_0$ to the open interval $(0, \infty)$. We set $\mathcal W^{\lambda} = (\hat{\Omega},\hat{\FF},\P^{\lambda})$ to be the countable product of these spaces
$$ \mathcal W^{\lambda} := \bigotimes_{i=1}^{\infty} \mathcal W_{-i}^{\lambda} \times \mathcal W_0^{\lambda} \times  \bigotimes_{i=1}^{\infty} \mathcal W_{i}^{\lambda}. $$
Any $\omega \in \hat{\Omega}$ is therefore given by a sequence $(\xi_i)_{i \in \Z}$ with $\xi_0 \geq 0$ and $\xi_i > 0$ for all $i \in \Z \setminus \{ 0 \}$. Then the map
$\Pi : \hat{\Omega} \to \R^{\infty}$ given by
$$\Pi(\omega) \equiv \Pi \left( \,(\xi_i)_{i} \,\right) := \{Y_l \mid l \in \Z \} \quad \text{ with } \quad
Y_l := \Phi(\omega)_l := \begin{cases} \sum_{i=0}^{l} \xi_i & \text{ for } l \geq 0\\ - \sum_{i=1}^{-l} \xi_{-i} &\text{ for } l < 0\end{cases}$$
defines a homogeneous Poisson process on $\R$ with rate $\lambda$. Finally, we modify the probability space $\mathcal W^{\lambda}$ and all related quantities by restricting $\hat{\Omega}$ to those $\omega\equiv (\xi_i)_i$ for which  $\sum_{i=0}^{\infty} \xi_i  = \infty$ and $\sum_{i=1}^{\infty} \xi_{-i} = \infty$. As argued above this removes a set of measure $0$. Thus the modified version also defines a
homogeneous Poisson process on $\R$ with rate $\lambda$ and we transfer our notation $\mathcal W^{\lambda}$, $\hat{\Omega}$, $\hat{\FF}$, $\P^{\lambda}$, $\Pi$ to the modified version.
\end{definition}

There is a somewhat confusing aspect about this construction. The distance between neighboring points $Y_l-Y_{l-1}$ is exponentially distributed except for $l=0$, because  $Y_0-Y_{-1}=\xi_{-1}+\xi_0$ is the sum of two independent exponentially distributed random variables and is therefore not distributed exponentially. However, this does not contradict the fact that for any time $t\in\R$ the time until the next jump attempt is exponentially distributed. The reason for this is the waiting time paradox,
 and we refer the reader to the end of 
Section~4.1 in~\cite{Kingman} for an explanation.

We now study the question raised at the beginning of this subsection, i.e.~to identify the map $\hat{\theta}_t$ that is induced on $\hat{\Omega}$ by shifting the origin of the real axis to $t$. First we notice that the map $\Phi: \omega \mapsto (Y_l)_{l \in \Z}$ that is implicit in Definition \ref{def:Pi} is a bijection between $\hat{\Omega}$ and the set of all strictly increasing sequences $(a_l)_{l \in \Z}$ with $a_{-1} < 0 \leq a_0$ and $\lim_{l\to\pm\infty} a_l = \pm\infty$. Note that the last property follows from the modification performed at the end of Definition \ref{def:Pi}. 
Now fix $t \in \R$. 
The translated sequence $(Z_l)_{l \in \Z}$ defined by $Z_l:=Y_l-t$ is again a strictly monotone sequence that is unbounded from above and below and we can therefore find some $l_0 \in \Z$ with $Z_{l_0-1} < 0 \leq Z_{l_0}$. Consequently, the shifted sequence
$(b_l)_{l \in \Z}$ with $b_l:=Z_{l+l_0}$ lies in the range of the map $\Phi$. Defining $\hat{\theta}_t(\omega):=\Phi^{-1}((b_l)_l)$ we have that the set $\Pi(\omega)$ shifted by $-t$ equals $\Pi(\hat{\theta}_t(\omega))$ as desired. 
Moreover, the just defined family $\{ \hat{\theta}_t\}_{t \in \R}$ satisfies properties (i)-(iii) of Definition \ref{def:rds} by construction. In order to see that all $\hat{\theta}_t$ preserve the measure $\P^{\lambda}$ one may proceed as follows: By the Mapping Theorem \cite[Section 2.3]{Kingman} the shifted map $\Pi':=\Pi-\{t\}$ is again a homogeneous Poisson process with rate $\lambda$. As argued in the paragraph above Definition \ref{def:Pi} the corresponding random variables $(\xi'_i)_{i \in \Z}$ are again iid and exponentially $\lambda$-distributed by the Restriction Theorem and by the Interval Theorem. Hence the distribution of $(\xi'_i)_{i \in \Z}$ is again governed by $\P^{\lambda}$.

After all these preparatory discussions we are finally ready to define our RDS.

\subsection*{Definition of the TASEP random dynamical system}\label{sec:rds}

We begin with the probability space $(\Omega,\FF,\P)$. It is essentially given by the $(n+1)$-fold product (cf.~Definition~\ref{def:Pi})
$$ \mathcal W^{\alpha} \times  \bigotimes_{k=1}^{n-1} \mathcal W^{h_k} \times   W^{\beta} \,.$$
Any $\omega$ then corresponds to $n+1$ stochastically independent point processes that we may represent by strictly increasing sequences $(T_{k,j}(\omega))_{j\in \Z}$ that are unbounded above and below. Here $k=0,\ldots,n$ denotes the lattice site of the random clock where $k=0$ represents the clock for particles entering the first site. Since the exponential distribution is absolutely continuous it is not hard to see that the event that there exist $k\neq k'$, $j$, $j'$ with $T_{k,j}(\omega)=T_{k',j'}(\omega)$ has zero probability and we remove this event from our probability space. This completes the definition of $(\Omega,\FF,\P)$.   

By construction, all jump times $T_{k,j}(\omega)$ are pairwise distinct for all $\omega \in \Omega$. Therefore there exist unique sequences  
$k_i=k_i(\omega)$ and $j_i=j_i(\omega)$, $i \in \Z$, with 
\[ T_{k_i,j_i}(\omega) < T_{k_{i+1},j_{i+1}}(\omega)\,, i \in \Z\,, \quad \text{and} \quad T_{k_{-1},j_{-1}} < 0 \;\leq T_{k_{0},j_{0}} \,.\]
We call the random sequence $(k_i)_i$ the {\em jump order sequence} with corresponding {\em jump time sequence} $t_i := T_{k_i,j_i}$.

The dynamics $\theta_t$ is defined by 
$\theta_t(\omega) \equiv \theta_t(\omega_0, \ldots, \omega_n) := (\hat{\theta}_t(\omega_0), \ldots, \hat{\theta}_t(\omega_n))$. We noticed above that $\{ \hat{\theta}_t\}_t$ satisfies properties (i)-(iii) of Definition \ref{def:rds} and argued that $\hat{\theta}_t$ leaves the probability measure invariant. All this carries over to $\theta_t$ acting on the product space restricted to the event of pairwise distinct jump times. Moreover, it is clear from the construction that for every $\omega \in \Omega$ and $t \in \R$ there exists $\Delta i (\omega, t) \in \Z$ such that the jump time and jump order sequences satisfy
\begin{equation}\label{eq:actionoftheta}
t_{i-\Delta i (\omega, t)}(\theta_t(\omega)) = t_i (\omega)-t\,,\quad k_{i-\Delta i (\omega, t)}(\theta_t(\omega)) = k_i (\omega) \quad \text{for all} \; i\in\Z \,.
\end{equation}

Finally, we can define the cocycle mapping $\varphi$ using $\tilde\varphi$ of \eqref{eq:tildephi} and the just defined sequences of jump order $(k_i)_i$ and jump times  $(t_i)_i$:
\[ \varphi(t,x,\omega) := \tilde\varphi(t,x,(t_i(\omega)),(k_i(\omega)))\,. \]
Let us check the requirements of Definition \ref{def:rds}. There is nothing to show for condition (4), because the state space is discrete. Condition (3) follows from the construction and (1) holds because $\tilde\varphi$ leaves $x$ unchanged for $t=0$ (there is no jump time $t_i(\omega)$ in $[0, t) = \emptyset$). The cocycle property (2) is an immediate consequence of \eqref{eq:actionoftheta}.

\subsection*{Non-Uniqueness of the TASEP random dynamical system} 

The construction of the random dynamical system for TASEP we just presented appears to be the most reasonable one from the point of view of physical intuition. This is why we use it in the remainder of this paper. However,  we would like to point out that the RDS is not uniquely determined by the transition probabilities as we illustrate now. For instance, in TASEP with chain length $n=3$, when the clock for site $1$ rings at time $t$ (i.e., when $T_{1,j}(\omega)=t$ for some $j\in\Z$), then this will trigger the transitions
\[ 100 \to 010 \quad \mbox{and} \quad 101 \to 011 \]
while the clock at site $2$ will trigger the transitions
\[ 010 \to 001  \quad \mbox{and} \quad 110 \to 101 \]
(for all other states, nothing will happen). Now, if the rates~$h_1$ and~$h_2$ are equal, then one may also define a random dynamical system
 in which the ringing of clock number 1 triggers the transitions
\[ 100 \to 010 \quad \mbox{and} \quad 110 \to 101 \]
while clock number $2$ triggers the transitions
\[ 010 \to 001  \quad \mbox{and} \quad 101 \to 011. \]
Since the rates of the two clocks are the same, this will yield a model with exactly the same state transition statistics as the one constructed above, but the resulting random dynamical systems differ: for the redefined dynamical system 
there is a positive probability that both $\varphi(1,100,\omega)=010$ and $\varphi(1,110,\omega)=101$ hold true, while for the system defined in the previous section such an $\omega$ does not exist. The latter statement can be verified by checking that there are no two states $x_1$, $x_2$ such that $x_1$ goes to or remains at $010$ and simultaneously $x_2$ goes to or remains at $101$ by the ringing of any of the four clocks. 

\section{Random attractors for finite state random dynamical systems}\label{sec:attractors}

\LG{As we have seen, TASEP can be formulated 
 as an RDS with a finite state space. In this section we present results for attraction and in particular random attractors of general RDS with finite state space. While these results are interesting in their own right, in the subsequent section we will in particular use them for TASEP.}

\subsection{Random attractors}

We use the following definitions of random attractors in the pullback and in the forward sense. We refer to \cite{ChKS02,Sche02} for a study of the difference between pullback and forward attraction. Here we limit ourselves to the definition of {\em global} random attractors. 

\begin{definition} A {\em random set} $C$ on a probability space $(\Omega,\FF,\P)$ is a measurable subset of $X\times\Omega$ with respect to the product $\sigma$-algebra of the Borel $\sigma$-algebra of $X$ and $\FF$. The $\omega$-section of a random set $C$ is for each $\omega\in\Omega$ defined by
\[ C(\omega) = \left\{ x\in X\,|\, (x,\omega)\in C\right\}.\]
The random set is called compact if every $C(\omega)$ is compact.
\end{definition}

\begin{definition} Let $(\theta,\varphi)$ be an RDS on a Polish space $X$. A 
compact random set $A\subset X\times\Omega$ that is strictly $\varphi$-invariant, i.e., 
\[ \varphi(t,A(\omega),\omega) = A(\theta_t\omega) \mbox{ for every } t\in\R_0^+ \mbox{ a.s.}, \]
is called a {\em global random pullback attractor}, if 
\[ \lim_{t\to\infty} \dist\big(\varphi(t, X, \theta_{-t}\omega), A(\omega)\big) = 0 \mbox{ a.s.}. \]
It is called a \emph{global random forward attractor} if 
\[ \lim_{t\to\infty} \dist\big(\varphi(t, X, \omega), A(\theta_t\omega)\big) = 0 \mbox{ a.s.}. \]
Here $\dist(A,B) := \sup_{a\in A}\inf_{b\in B}d(a,b)$. 
\end{definition}

In our case the finite state space $X$ is equipped with the discrete topology and we may therefore use the distance defined by $d(x_1,x_2)=1$ if $x_1\ne x_2$ and $d(x_1,x_2)=0$ if $x_1= x_2$. This implies for subsets $A$, $B \subset X$ that $d(A,B)=0$ if $A\subset B$ and $d(A,B)=1$, otherwise.
For the construction of the attractor, we use that for each $s > 0$ the cocycle property implies
\begin{equation} \varphi(t+s,X,\theta_{-t-s}\omega) = \varphi(t,\varphi(s,X,\theta_{-t-s}\omega),\theta_{-t}\omega) \subseteq \varphi(t,X,\theta_{-t}\omega), \label{eq:cocyclsts}\end{equation}
so  the set $\varphi(t,X,\theta_{-t}\omega)$ is decreasing in $t$ w.r.t.\ set inclusion. Hence, we can define its set valued limit via
\begin{equation} A(\omega) := \bigcap_{t\ge 0} \varphi(t,X,\theta_{-t}\omega). \label{eq:Adef}\end{equation}

\LG{
\begin{theorem} Consider an RDS with  a  finite state space. Then $A(\omega)$ from \eqref{eq:Adef} is nonempty for each $\omega\in\Omega$ and defines both a global random pullback attractor and a global random forward attractor. Moreover, for each $\omega\in\Omega$ there exists $T(\omega)>0$ such that $\varphi(t,X,\theta_{-t}\omega) = A(\omega)$ for all $t\ge T(\omega)$ and for almost every $\omega\in\Omega$ there exists $\widehat T(\omega)>0$ such that $\varphi(t,X,\omega) = A(\theta_{t}\omega)$ for all $t\ge \widehat T(\omega)$
\label{thm:rattr}\end{theorem}  }
\begin{proof}
\LG{
For $t\ge 0$ and $\omega\in\Omega$ define $B_t(\omega): = \varphi(t,X,\theta_{-t}\omega)$. From \eqref{eq:cocyclsts} we obtain that for all $t,s\ge 0$, $x\in X$ and $\omega\in\Omega$
\begin{equation} B_{t+s}(\omega) =  \varphi(t,B_s(\theta_{-t}\omega),\theta_{-t}\omega) \subseteq B_t(\omega). \label{eq:Bts}\end{equation}
Since these sets are finite, this inclusion implies that $t\mapsto B_t(\omega)$ can change its value only finitely many times, implying that $A(\omega) = \bigcap_{t\ge 0} B_{t}(\omega)$ equals $B_t(\omega)$ for all sufficiently large $t\ge 0$ and is thus in particular nonempty. Let $T(\omega)>0$ be the infimal time for which $A(\omega) = B_t(\omega)$ holds for all $t\ge T(\omega)$. Then
\[\dist(\varphi(t,X,\theta_{-t}\omega), A(\omega)) = \dist(B_t(\omega), A(\omega)) = 0 \]
for all $t\ge T(\omega)$, i.e., finite time pullback attraction.
}

\LG{
Next we prove $\varphi$-invariance of $A$, i.e., $\varphi(t,A(\omega),\omega) = A(\theta_t\omega)$ for all $t\ge 0$ and all $\omega\in\Omega$. To this end, fix $t\ge 0$ and $\omega\in\Omega$ and choose $s\ge \max\{ T(\omega), T(\theta_t\omega)-t
\}$. Then we get $A(\omega)=B_s(\omega)$ and $A(\theta_t\omega) = B_{t+s}(\theta_t\omega)$. Together with the first identity in \eqref{eq:Bts}, applied with $\theta_t\omega$ in place of $\omega$, this yields
\[ \varphi(t,A(\omega),\omega) = \varphi(t,B_s(\omega),\omega) = B_{t+s}(\theta_t\omega)=A(\theta_t\omega).\]
Together, this shows that $A$ is a global random pullback attractor and that pullback attraction happens in finite time $T(\omega)$ for each $\omega\in\Omega$.
}

\LG{
In order to see that $A$ is also a global forward attractor, for every $T>0$ define the set $\Omega_T:= \{\omega\in\Omega\,|\, T(\omega) \le T\}$. Since $T(\omega)$ is finite for every $\omega$, for every $p\in(0,1)$ there exists $T_p>0$ such that $\P(\Omega_{T_p})>p$ holds. Now, consider the set $\widehat \Omega_{T_p} := \theta_{-T_p}\Omega_{T_p}$. Since $\theta_t$ is measure-preserving, we obtain that $\P(\widehat \Omega_{T_p})>p$. Now, each $\hat\omega\in\widehat \Omega_{T_p}$ is of the form $\hat\omega=\theta_{-T_p}\omega$ for some $\omega\in \Omega_{T_p}$, i.e., $\omega= \theta_{T_p}\hat\omega$. Hence, for each $\hat\omega\in\widehat\Omega_{T_p}$ we have
\[ \dist(\varphi(T_p,X,\hat\omega), A(\theta_{T_p}\hat\omega)) = \dist(\varphi(T_p,X,\theta_{-T_p}\omega),A(\omega)) = 0.\]
For all $t\ge T_p$ we obtain
\[ \varphi(t,X,\hat\omega) = \varphi(t-T_p, \varphi(T_p,X,\hat\omega), \theta_{T_p}\hat\omega) = \varphi(t-T_p, A(\theta_{T_p}\hat\omega), \theta_{T_p}\hat\omega) = A(\theta_{t}\hat\omega),\]
where we used $\varphi$-invariance of $A$ in the last step. This implies 
\[ \dist(\varphi(t,X,\hat\omega), A(\theta_{t}\hat\omega)) = 0 \]
for all $t\ge T_p$ and thus forward attraction in finite time $T_p$ with probability larger than $p$. Since $p\in(0,1)$ is arbitrary, this implies forward attraction to $A$ in finite time with arbitrarily large probability and thus almost sure forward attraction in finite time.}
\end{proof}

\LG{
Theorem \ref{thm:rattr}  shows in particular   that for an~RDS with a finite state space the concept of pullback and forward random attractors coincide. This is in contrast to the general case, where (simple) SDEs are known that exhibit invariant random sets that are pullback but not forward random attractors and vice versa, cf.\ \cite{Sche02}. The intuitive reason for this is that, on a finite state space, once a solution gets sufficiently close to an invariant set then it must already be inside the invariant set. Thus, it can never leave the set again due to its $\varphi$-invariance, i.e., it gets ``trapped''. In contrast to this, on an infinite state space solutions can leave every neighbourhood of an invariant set, no matter how small it is. This effect was exploited for constructing the examples in \cite{Sche02}.
}

\begin{remark} \LG{(i) We note that there is an asymmetry between the statements for forward and pullback attraction in Theorem \ref{thm:rattr}: while pullback attraction holds for all $\omega\in\Omega$, forward attraction holds only for almost all $\omega\in\Omega$. This cannot be strengthened, 
as we demonstrate using~TASEP in Remark~\ref{rem:nonsync}(ii), below.
}

\LG{
(ii) While Theorem \ref{thm:rattr} shows that random attractors both in the forward and in the pullback sense always exist for  a finite state RDS, they may not necessarily carry much information. One example for this is the random walk on $\Z_3$, see also \cite[Example 1.8]{LevP17}. We can formulate this process as an RDS in the same way as TASEP, using the state space $X=\{0,1,2\}$, the indices $k\in\{1,2\}$ (i.e., two ``random clocks'') and the map 
\[ f(x,k) = x+k \mod 3.\]
For this process, regardless of the order of the ringing of the clocks in the associated Poisson processes, one easily sees that for any two initial conditions $x_1$, $x_2\in X$ with $x_1 = x_2 + j \mod 3$, the solutions satisfy $\varphi(t,x_1,\omega) = \varphi(t,x_2,\omega) + j \mod 3$ for all $t\ge 0$. This implies that $\varphi(t,X,\omega) = X$ for all $\omega\in\Omega$ and all $t\ge 0$, yielding $A(\omega)=X$ for all $t\ge 0$. Hence, the sets forming the random attractor are the whole state space and the attractor becomes ``trivial'' in the sense that it does not give any information about the long term behavior except the trivial information that it is contained in $X$.} 
\label{rem:rw3}
\end{remark}

\subsection{Random attractors consisting of single trajectories}

\LG{
As we have seen in the last example, the random attractor need not yield useful information. Generally speaking, the random attractor gives  more information 
about the long time behavior of the RDS when the sets~$A(\omega)$ are small.  The most informative case is when  the random attractor consists of a single trajectory almost surely. In this case, the 
long time behavior is almost surely independent of the initial condition. The following theorem gives necessary and sufficient conditions for this property to hold.}
For its formulation we define the sets 
\begin{equation} \Gamma(t,t_0):= \{\omega\in\Omega\,|\,\varphi(t,x_1,\theta_{t_0}\omega)=\varphi(t,x_2,\theta_{t_0}\omega) \mbox{ for all } x_1,x_2\in X\}.\label{eq:gammadef}\end{equation}

\begin{theorem} \LG{Consider an RDS with a finite state space. \TK{Then statements (2)-(5) below are all equivalent to each other and imply statement (1).}} 
\LG{
\begin{enumerate} 
\item[(1)] There exists $t>0$ such that 
\[ \P(\Gamma(t,0)) > 0. \]
\item[(2)] For any $t_0\in\R$ it holds that
\[ \lim_{t\to\infty} \P(\Gamma(t,t_0)) = 1 \quad \mbox{and} \quad \lim_{t\to\infty} \P(\Gamma(t,t_0-t)) = 1 \]
and the rate of convergence is independent of $t_0$.
\item[(3)] For any $t_0\in\R$ 
\[ \P(\{\omega\in\Omega\,|\, \varphi(t,X,\theta_{t_0}\omega) \mbox{ is a singleton for some } t\ge 0\}) = 1.\]
\item[(4)] For any $t_0\in\R$ 
\[ \P(\{\omega\in\Omega\,|\, \varphi(t,X,\theta_{t_0-t}\omega) \mbox{ is a singleton for some } t\ge 0\}) = 1.\]
\item[(5)] The sets $A(\omega)$ from \eqref{eq:Adef} are singletons for almost all $\omega\in\Omega$.
\end{enumerate}
}
\LG{\TK{If in addition  
}
the solutions $\varphi(t,x,\cdot)$ are stochastically independent on non-overlapping intervals, i.e.,~$\varphi(t,x,\cdot)$ for $t\in[t_1,t_2)$ is independent of $\varphi(s,x,\cdot)$ for $s\in[s_1,s_2)$ if $[t_2,t_2) \cap [s_1,s_2) = \emptyset$\TK{, then statement~(1) is equivalent to statements~(2)-(5),
 and the convergence in statement~(2) has an exponential rate.} 
}
\label{thm:singleton}
\end{theorem}
\begin{proof}
\TK{We first show for a general finite state RDS the implication (2) $\Rightarrow$ (1) as well as the equivalences (2) $\Leftrightarrow$ (3), (2) $\Leftrightarrow$ (4), and (4) $\Leftrightarrow$ (5). The proof is then completed by demonstrating the implication (1) $\Rightarrow$ (2) with an exponential rate of convergence under the additional assumption of stochastic independence of the RDS.}

\TK{
(2) $\Rightarrow$ (1): This is obvious using the first part of statement (2).
}

\TK{
(2) $\Leftrightarrow$ (3): Since $\varphi(t,X,\theta_{t_0}\omega)$ is a singleton if and only if $\varphi(t,x_1,\theta_{t_0}\omega)=\varphi(t,x_2,\theta_{t_0}\omega)$  holds for all $x_1,x_2\in X$ we have 
\begin{equation}\label{eq:2to3}  
\bigcup_{t > 0} \, \Gamma (t, t_0) \; = \; \{\omega\in\Omega\,|\, \varphi(t,X,\theta_{t_0}\omega) \mbox{ is a singleton for some } t\ge 0\} \,.
\end{equation}
Note in addition that the sets $\Gamma (t, t_0)$ are monotonically increasing in $t$ as for all $\omega\in\Omega$ the identity $\varphi(t,x_1,\omega)=\varphi(t,x_2,\omega)$ ensures that
$\varphi(s,x_1,\omega)=\varphi(s,x_2,\omega)$ for all $s>t$ by the cocycle property in Definition~\ref{def:rds}. By monotone convergence we therefore have
\[ \lim_{t\to\infty} \P(\Gamma(t,t_0))  \; = \; \P(\{\omega\in\Omega\,|\, \varphi(t,X,\theta_{t_0}\omega) \mbox{ is a singleton for some } t\ge 0\})\]
which implies the equivalence between (3) and the first statement of (2). Observe that the measure perserving flow $\theta_t$ is a bijection from the set $\Gamma(t,t_0)$
onto the set $\Gamma(t,t_0-t)$ so that $\P(\Gamma(t,t_0)) = \P(\Gamma(t,t_0-t))$. This shows that the two statements formulated in~(2) are equivalent.
}

\TK{
(2) $\Leftrightarrow$ (4): The relation corresponding to \eqref{eq:2to3} reads
\begin{equation}\label{eq:2to4}
\bigcup_{t > 0} \, \Gamma (t, t_0-t) \; = \; \{\omega\in\Omega\,|\, \varphi(t,X,\theta_{t_0-t}\omega) \mbox{ is a singleton for some } t\ge 0\} \,.
\end{equation}
Following the proof of the equivalence (2) $\Leftrightarrow$ (3) we only need to argue the monotonicity of the sets $\Gamma (t, t_0-t)$ with respect to $t$.
Replacing $\omega$ by $\theta_{t_0}\omega$ the monotonicity is a consequence of equation \eqref{eq:cocyclsts} that holds for $s,t >0$.
}

\TK{
(4) $\Leftrightarrow$ (5): The definition of the set $A(\omega)$ in~\eqref{eq:Adef} together the finiteness of the state space gives (cf.~the proof of~Theorem \ref{thm:rattr}) 
\begin{equation}\label{eq:4to5}
 \bigcup_{t > 0} \, \Gamma (t,-t)  \; = \;  \{\omega\in\Omega\,|\,A(\omega) \mbox{ is a singleton}\}\,.
\end{equation}
For every $t_0\in\R$ the measure preserving map $\theta_{t_0}$ is a bijection from the set $\bigcup_{t > 0} \, \Gamma (t, t_0-t)$ onto the set $\bigcup_{t > 0} \, \Gamma (t,-t)$. Therefore the sets on the right-hand-sides of relations \eqref{eq:2to4} and \eqref{eq:4to5} have the same probability.
}

\TK{From now on we also assume the stochastic independence of the RDS.}

(1) $\Rightarrow$ (2): \TK{Recall first that} the identity $\varphi(t,x_1,\omega)=\varphi(t,x_2,\omega)$ ensures that
$\varphi(s,x_1,\omega)=\varphi(s,x_2,\omega)$ for all $s>t$ \TK{and for all $\omega\in\Omega$}. This   yields
\begin{equation} \Gamma(t_3,t_1) \supset \Gamma(t_3,t_2) \cup \Gamma(t_2,t_1). \label{eq:gammasupset}\end{equation}
This in particular implies that the map $t\mapsto \P(\Gamma(t,0))$ is increasing, hence (1) implies that there exist $\delta>0$ and $\hat t>0$ such that $\P(\Gamma(t,0)) \ge \delta$ for all $t\ge \hat t$. Since the flow $\theta_{t_0}$ is measure preserving, we obtain $\P(\Gamma(t,t_0)) = \P(\Gamma(t,0))$ for all $t_0\in\R$. This implies that 
\begin{equation} \P(\Gamma(t,t_0))\ge \delta\label{eq:gammadelta}\end{equation}
for all $t\ge \hat t$ and all $t_0\in\R$, and that it is sufficient to prove $\lim_{t\to\infty} \P(\Gamma(t,0)) = 1$. Moreover, because of the monotonicity of $t\mapsto \P(\Gamma(t,0))$ it suffices to prove this convergence for a suitable sequence $t_m\to\infty$.

To this end, observe that for the complements $A^C := \Omega \setminus A$, relation \eqref{eq:gammasupset} implies 
\begin{equation} \Gamma(t_3,t_1)^C \subset \Gamma(t_3,t_2)^C \cap \Gamma(t_2,t_1)^C \label{eq:gammasubset}\end{equation}
for all $t_1<t_2<t_3$. We fix an arbitrary $\Delta t>\hat t$ and let $t_m=m\Delta t$. From \eqref{eq:gammadelta} it follows that 
\[ \P(\Gamma(\Delta t, t_m)) \ge \delta > 0, \]
implying
\[ \P(\Gamma(\Delta t,t_m)^C) \le 1-\delta <1. \]
Recall that $\delta$ only depends on $\Delta t$ and not on $m$. By the assumption on the RDS, the sets $\Gamma(\Delta t,t_m)$ and thus the sets $\Gamma(\Delta t,t_m)^C$ are stochastically independent for different $m$. Thus, using \eqref{eq:gammasubset} we obtain
\[ \P(\Gamma(t_{m},0)^C) \le \P\left(\bigcap_{l=1}^m\Gamma(\Delta t,t_{l-1})^C\right) = \prod_{l=1}^m \P(\Gamma(\Delta t,t_{l-1})^C) \le (1-\delta)^m \to 0\]
as $m\to\infty$. This implies $\P(\Gamma(t_{m},0))\to 1$ \TK{at an exponential rate, proving} the claim.
\end{proof}

\begin{remark} 
\LG{(i) Property (3) in Theorem \ref{thm:singleton} formalizes the fact that all trajectories synchronize almost surely. Hence, Theorem \ref{thm:singleton} in particular shows that almost sure synchronization is equivalent to the sets $A(\omega)$ being singletons almost surely.}

\LG{
(ii) In general we cannot expect that \TK{properties (2)-(5)} in Theorem \ref{thm:singleton} hold for all $\omega\in\Omega$ (as opposed to {\em for almost all $\omega$}). Remark \ref{rem:nonsync}(i), below, illustrates this for TASEP.
}
\end{remark}

\section{Random attraction in TASEP}\label{sec:TASEPattractors}

\LG{
We now apply the results of the previous section to TASEP. To this end, first note that both Theorem~\ref{thm:rattr} and all the equivalent statements in Theorem~\ref{thm:singleton} apply to the TASEP RDS, since the state space is finite and the paths of the homogeneous Poisson process are independent on non-overlapping intervals, implying the same for the solutions $\varphi$. 
}

\LG{
Hence, the sets $A(\omega)$ from \eqref{eq:Adef} define both a global forward and a global pullback random attractor. In order to show that the sets are singletons almost surely, we will verify property~(1) in Theorem~\ref{thm:singleton}. 
}

The proof of this property relies on the jump order sequences  
$(k_i)_i$ that were introduced in Section \ref{sec:rds}. 
Given a bounded interval $I=[\tau_1,\tau_2)$ of positive length then for all $\omega \in \Omega$ there exist only finitely many $i \in \Z$ with $t_i(\omega) \in I$, because the jump times do not accumulate on the real line by construction.
Therefore only finite jump order sequences can be be realized in $I$. In turn, we now argue that for
any prescribed finite tuple of sites $(k_1, \ldots, k_m)$ of any given length $m$, there is a positive probability that this tuple equals the section of the jump order sequence which corresponds to the jump times in the interval $I=[\tau_1,\tau_2)$. Indeed, devide 
$I$ into $m$ subintervals $I_i=(q_{i-1},q_{i})\subset I$, $\tau_1=q_0<q_1<\ldots<q_m = \tau_2$ and 
denote by $\Pi_k$ the point process that is associated with the probability space $\mathcal W^{h_k}$ that is used in the construction of~$\Omega$ above 
with~$h_0:=\alpha$ and~$h_n:=\beta$. Since the point processes~$\Pi_k$ are independent one can compute from Definition~\ref{def:Ppp} the probability that for each $i=1, \ldots, m$ we have $N_{k_i}(I_i)=N_{k_i}(\bar{I_i})=1$ and that $N_{k}(\bar{I_i})=0$ for all $k \in \{0, \ldots, n\} \setminus \{k_i\}$ where the random variables $N_k(A) := \# (\Pi_k \cap A)$ are defined according to Definition~\ref{def:Ppp}. This probability is a finite product of positive numbers and therefore positive as we have claimed.
Due to homogeneity of the Poisson processes, the probability for a particular jump order sequence to occur depends only on the length $\tau_2-\tau_1$ of the interval $I$ and not on the concrete values of $\tau_1$ and $\tau_2$. Moreover, on two non-overlapping intervals $[\tau_1,\tau_2)$ and $[\tau_3,\tau_4)$, where~$\tau_1<\tau_2\le \tau_3<\tau_4$, the jump order sequences are stochastically independent; this follows from the fact that $\Pi_k \cap [\tau_1,\tau_2)$ and $\Pi_{k'} \cap [\tau_3,\tau_4)$ are independent both for $k = k'$ due to condition (ii) of  Definition \ref{def:Ppp} and for $k\neq k'$ which, by construction,  holds even for the unrestricted point processes~$\Pi_k$ and~$\Pi_{k'}$.

Based on these observations, we can state the following result 
for the set~$\Gamma(t,t_0)$ from~\eqref{eq:gammadef}.

\begin{lemma} 
In TASEP, for any $t>0$ there exists a $\delta>0$ such that \[ \P(\Gamma(t,0)) \ge \delta. \]
\label{lemma:fintime}
\end{lemma}
\begin{proof}
According to our previous considerations, the probability that on the interval $[0,t)$ the jump order sequence 
\[  n, \;\; n-1, n, \;\; n-2, n-1, n, \;\; n-3, \ldots,n, \;\; \ldots,
 \;\;1,2, \ldots,n \]
occurs is equal to some $\delta>0$.
Now, for all $\omega\in\Omega$ generating this sequence it is easily seen that $\varphi(t,x,\omega) = (0,\ldots,0)$ for all $x\in X$. This shows the claim.
\end{proof}

\begin{remark}\label{rem:fintime}
\LG{We note that the statement of Lemma \ref{lemma:fintime} is stronger than property (1) of Theorem~\ref{thm:singleton}, because we obtain the desired inequality $\P(\Gamma(t,0)) >0$ for all $t>0$. }
\end{remark}

\begin{remark} The probability bound $\delta>0$ that follows from the proof of Lemma \ref{lemma:fintime} is clearly not optimal. \TK{In order to obtain upper bounds on mixing times in the asymptotic regime of large lattice sizes a number of techniques have been developed to obtain good lower bounds on the probability  
$\P(\Gamma(t,0))$ for simple exclusion processes, including  TASEP, with open boundaries and constant jump rates in the interior, see \cite{Gantert_etal_2020} and the references therein.}
\end{remark}

Using Lemma \ref{lemma:fintime}, we can now state the following theorem. 

\begin{theorem} In TASEP, the sets $A(\omega)$ from \eqref{eq:Adef} define a forward and a pullback global random attractor, where almost surely the set $A(\omega)$ is a singleton.
\end{theorem}
\begin{proof}\LG{
This follows from combining Theorem \ref{thm:rattr}, Theorem \ref{thm:singleton}, and Lemma \ref{lemma:fintime}.}
\end{proof}

\begin{remark}\label{rem:nonsync}
(i) \TK{Statement} (3) in Theorem \ref{thm:singleton} in particular implies that any two solutions of TASEP synchronize almost surely after sufficiently large time. However, this almost sure identity does not exclude the existence of non-trivial jump time sequences for which $\varphi(t,x_1,\theta_{t_0}\omega)$ and $\varphi(t,x_2,\theta_{t_0}\omega)$ never coincide. Consider TASEP with, e.g., chain length $n=3$ and initial conditions $x_1=110$ and $x_2=000$. Any jump time sequence that on $[t_0,\infty)$ generates the periodic jump order sequence $$(k_i(\omega))=(2, 1, 0, 1, 2,3 ; 2,1,0,1,2,3; 2,1,0,1,2,3; \ldots)$$ yields the two solutions 
\[ \begin{array}{lcc}t& \varphi(t,x_1,\theta_{t_0}\omega) & \varphi(t,x_2,\theta_{t_0}\omega)\\[2ex]
t_0 & 110 & 000\\
t_1(\omega) & 101 & 000\\
t_2(\omega) & 011 & 000\\
t_3(\omega) & 111 & 100\\
t_4(\omega) & 111 & 010\\
t_5(\omega) & 111 & 001\\
t_6(\omega) & 110 & 000,
\end{array}\]
where the table shows the values of the solutions right after the jump times $t_i(\omega)$, $i\ge 1$, which are numbered such that $t_1(\omega)$ is the first jump time after the initial time $t_0$. The periodicity of these solutions 
implies that  the two solutions never coincide. However, \TK{statement}~(3) in
 Theorem~\ref{thm:singleton} implies that the set of $\omega$ corresponding to such jump time sequences must have measure $0$.

\LG{(ii) This example can also be used to show that in TASEP forward attraction of $A$ does not hold for every $\omega\in\Omega$. To this end, consider an $\omega$ for which $A(\omega)$ is a singleton. Since~\eqref{eq:Adef} implies that $A(\omega)$ only depends on the jumps for t < 0, we may choose the jumps for $t \ge 0$ arbitrarily without changing $A(\omega)$. Particularly, we may choose $\omega$ such that the jump sequence for $t \ge 0$ generates the periodic non-synchronizing trajectories from part (i) of this remark. This means that $\varphi(t,X,\omega)$ contains at least two points for any $t>0$. However, since $A(\omega)$ is a singleton, by invariance $A(\theta_t\omega) = \varphi(t,A(\omega),\omega)$ must be a singleton for all $t>0$. Hence, the
 forward convergence property~$\lim_{t\to\infty} \dist\big(\varphi(t, X, \omega), A(\theta_t\omega)\big) = 0$ cannot hold for this $\omega$. }
\end{remark}

\section{\LG{Relation to the dynamics of the distribution}}\label{sec:distributions}

\TK{
In the theory of Markov processes a central object of study is the evolution of the probability distribution on the state space that is induced by the process and which is governed by the master equation.
Similarly, it is natural to associate with a random dynamical system $(\theta, \varphi)$ (see Definition \ref{def:rds}) such an evolution. Starting at time $t_0=0$ in a state $x\in X$ we let~$P^t_{x}$ denote the pushforward of the probability measure $\mathbb P$ under the map $\omega \mapsto \varphi(t, x, \omega)$, $t\geq 0$. Clearly,  $P^t_{x}$ is  a probability distribution on $X$. In the case of a finite state space the probability of each state $z \in X$ is simply given by
$$
(P^{t}_x)(\{z\}) = \P(\{ \omega \in \Omega \, | \, \varphi(t,x,\omega)=z)\} ) \,.
$$
As in the case of Markov processes one may extend this definition to arbitrary initial probability distributions $\mu$ on $X$ via $P^{t}_\mu := \sum_{x \in X} \mu(\{x\}) P^{t}_x$.
A stationary distribution for the RDS is a probability distribution $\pi$ on $X$ that is invariant under the evolution, i.e., it satisfies $\pi = P^{t}_\pi$ for all $t\geq 0$. Note that in the special case  where  the RDS is constructed from a homogeneous Markov process on a finite state space, using the procedure that we made explicit for TASEP in Section~\ref{sec:TASEP}, the just presented definitions of the evolution of probability distributions and of the stationary measure agree with those used in the theory of Markov processes, see e.g., Section~20.1 below equation (20.3) in~\cite{LevP17}.
}

\TK{
We are now ready to formulate the main result of this section which states that almost sure synchronization implies that for all initial distributions $\mu$ and $\nu$ the distance between the evolving probability distributions 
$P^{t}_\mu$ and $P^{t}_\nu$, measured in the total-variance metric, tends to zero as time tends to infinity. If, in addition, it is known that a stationary measure exists, then one may conclude that this stationary distribution is a global attractor. The proof of this result uses that the RDS provides a grand coupling between probability distributions $P^{t}_x$ and $P^{t}_y$. The core of our argument is taken from the proof of Theorem~5.4 in~\cite{LevP17} for Markov chains. It is contained in the following Proposition~\ref{prop:coupling_estimate}. We also reproduce its short proof for the convenience of the reader.
}
\begin{proposition} \TK{Consider an RDS on a finite state space~$X$. Recall
 the definition of the set~$\Gamma(t, t_0)$ in~\eqref{eq:gammadef}. Then for all probability distributions $\mu$ and $\nu$ on $X$ we have $$\| P^{t}_\mu -P^{t}_\nu \|_{TV} \leq 1- \P (\Gamma (t,0)) \text{ for all } t\geq 0.$$
}\label{prop:coupling_estimate}
\end{proposition}
\begin{proof}
\TK{
For each $x,y \in X$ the distribution of the random variables $X_t:= \varphi(t,x,\cdot)$ and $Y_t:= \varphi(t,y,\cdot)$ is  given by $P^{t}_x$ and $P^{t}_y$ respectively.
In other words, the random variables $X_t$ and $Y_t$ provide a \emph{coupling} of the measures $P^{t}_x$ and $P^{t}_y$, see e.g.,~\cite[Section~4.2]{LevP17}. Denoting by $\bigtriangleup$ the symmetric difference operating on sets, we have for each subset $A \subset X$ that
\begin{equation}\label{eq:basic_estimate}
\left| P^{t}_x(A)-P^{t}_y(A) \right| \; \leq \; \P(\{X_t \in A\} \bigtriangleup \{Y_t \in A\}) \; \leq \; \P (X_t \neq Y_t) \; \leq \; 1- \P (\Gamma (t,0))\,.
\end{equation}
For any probability distribution $\mu$ on $X$  the value of $P^{t}_\mu(A)$ is bounded below and above by the minimal and maximal value that $P^{t}_x(A)$ attains as $x$ varies over $X$ respectively. It therefore lies in an interval the length of which is bounded above by $1- \P (\Gamma (t,0))$ due to estimate \eqref{eq:basic_estimate}. Therefore we obtain the bound $| P^{t}_\mu(A) -P^{t}_\nu(A)| \leq 1- \P (\Gamma (t,0))$ for all distributions $\mu, \nu$ on $X$ irrespective of the chosen set $A$. This proves the claim.
}
\end{proof}

\begin{theorem}
\TK{Consider an RDS on a finite state space $X$ that synchronizes almost surely in the sense that one and hence all of the statements (2)-(5) in~Theorem~\ref{thm:singleton} hold true. Then we have
\begin{equation}\label{eqn:contractionT}
\sup \{ \| P^{t}_\mu -P^{t}_\nu \|_{TV}  \; | \; \mu, \nu \mbox{ are probability distributions on $X$} \}  \; \longrightarrow \; 0\,, \; \mbox{ as $t \to \infty$}\,.
\end{equation}
If in addition a stationary measure $\pi$ exists for the RDS then $\pi$ is a global attractor with uniform convergence
\begin{equation}\label{eqn:convergenceT}
\sup \{ \| P^{t}_\mu - \pi \|_{TV}  \; | \; \mu \mbox{ is a  probability distribution on $X$} \}  \; \longrightarrow \; 0\,, \; \mbox{ as $t \to \infty$}\,.
\end{equation}
For both results \eqref{eqn:contractionT} and \eqref{eqn:convergenceT} the convergence has an exponential rate if the RDS is stochastically independent in the sense that is formulated in the statement of~Theorem~\ref{thm:singleton}. 
}
\label{thm:RDSdistrT}\end{theorem}
\begin{proof}
\TK{Claim \eqref{eqn:contractionT} is a direct consequence of Proposition \ref{prop:coupling_estimate} and of Theorem~\ref{thm:singleton}. Application of \eqref{eqn:contractionT} to the special case $\nu = \pi$ and using the defining relation $\pi = P^{t}_\pi$ for the stationary measure $\pi$ we immediately obtain the convergence result \eqref{eqn:convergenceT}.
}
\end{proof}

\TK{
An example where the evolution of probability distributions is contractive, but where no stationary measure exists is given by~TASEP with transition rates that are not constant but depend periodically on time, see~\cite{entrain_master} and the discussion contained in the last paragraph of Section~\ref{sec:extensions}.
}

\TK
{If the RDS is constructed from a finite state Markov process with constant transition rates then the existence of a stationary measure is guaranteed and almost sure synchronization implies \eqref{eqn:convergenceT}. Therefore the stationary measure is unique in this case and, consequently, the underlying Markov process has  a single essential communicating class, see e.g.,~\cite[Proposition 1.29]{LevP17}. 
Note, however, that the example of a random walk on the triangle $\mathbb Z_3$ presented in Remark~\ref{rem:rw3}~(ii) defines an irreducible Markov process with a unique and globally attracting stationary measure but the corresponding RDS does not synchronize as the attractor surely consists of all three states of the state space. This shows that almost sure synchronization is a stronger property for homogeneous finite state Markov processes than having a unique stationary measure.
}

\section{Numerical simulations of synchronization}\label{sec:NumSim}

\TK{To visualize  the random attraction that we have proved for~TASEP,
 we have simulated three random jump time sequences.} Each  one
was used to simulate two TASEP trajectories with different initial conditions. The resulting trajectories are shown in Figure \ref{fig:sim}, where each column corresponds to one of the jump time sequences and the states of the two trajectories with different initial conditions are shown on top of each other for times $t=0,10,\ldots,100$. The parameters were chosen as $n=20$ sites and $\alpha=\beta=h_k=1$ for all rates.

\begin{figure}
\begin{minipage}{1.4cm}$t=0$\\[6mm]
\end{minipage}
\includegraphics[trim=24 110 24 160, clip, width=4cm]{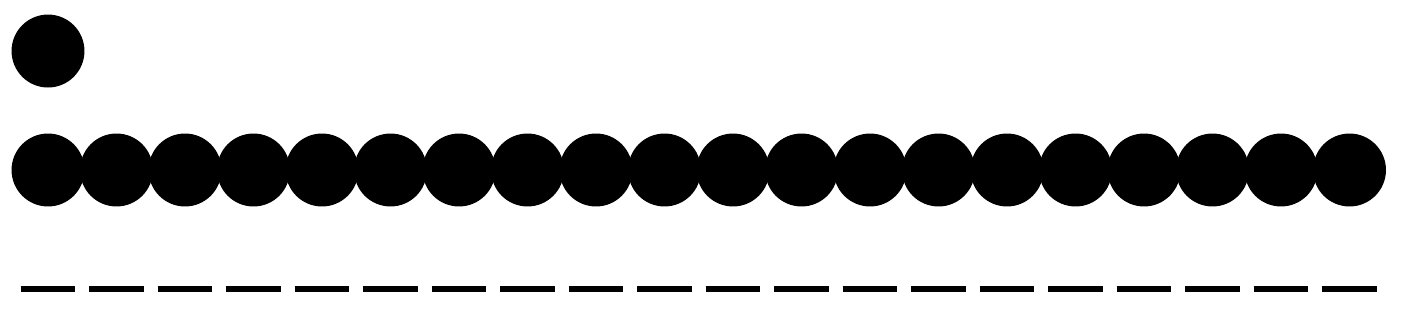}\qquad
\includegraphics[trim=24 110 24 160, clip, width=4cm]{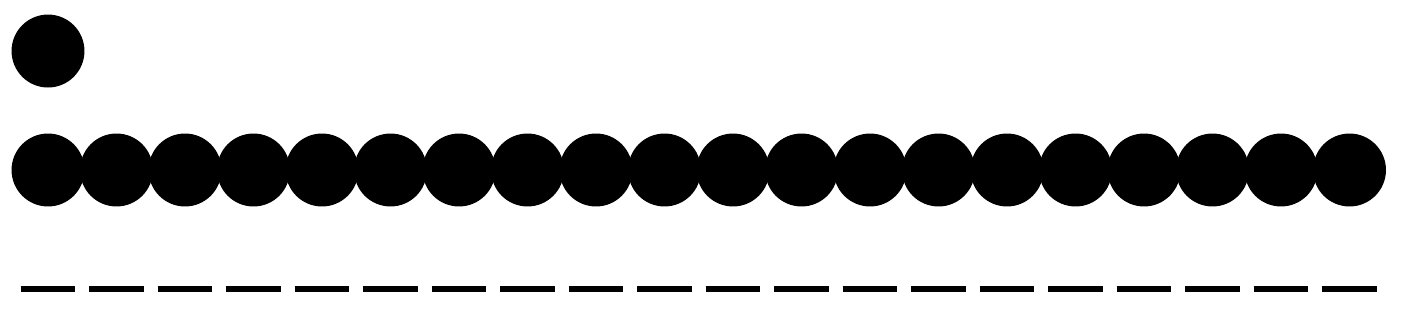}\qquad
\includegraphics[trim=24 110 24 160, clip, width=4cm]{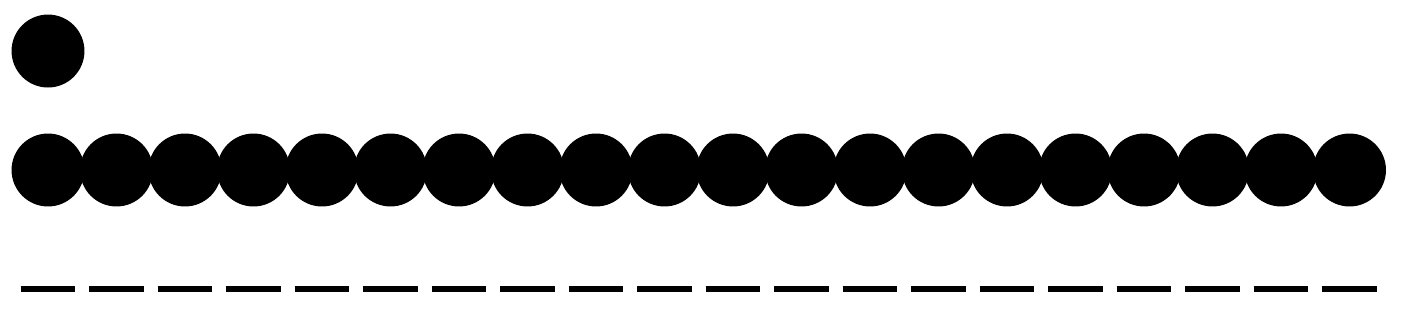}

\begin{minipage}{1.4cm}$t=10$\\[6mm]
\end{minipage}
\includegraphics[trim=24 110 24 160, clip, width=4cm]{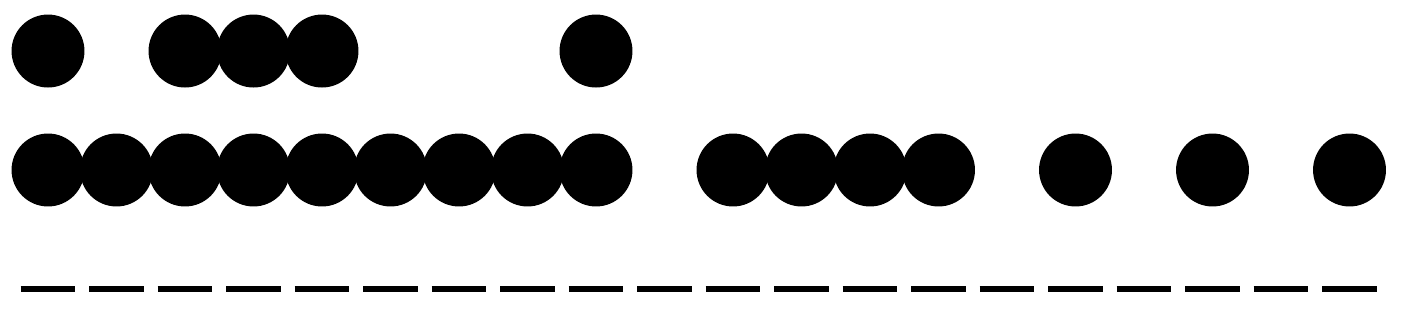}\qquad
\includegraphics[trim=24 110 24 160, clip, width=4cm]{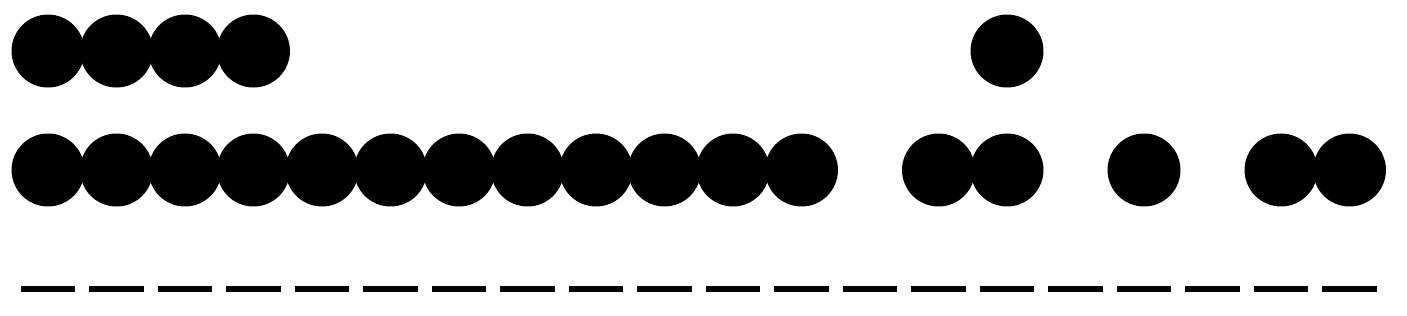}\qquad
\includegraphics[trim=24 110 24 160, clip, width=4cm]{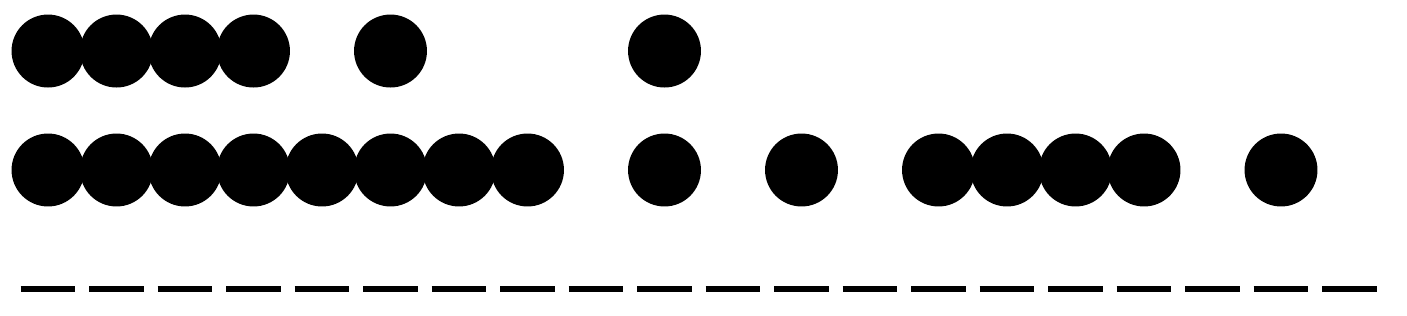}

\begin{minipage}{1.4cm}$t=20$\\[6mm]
\end{minipage}
\includegraphics[trim=24 110 24 160, clip, width=4cm]{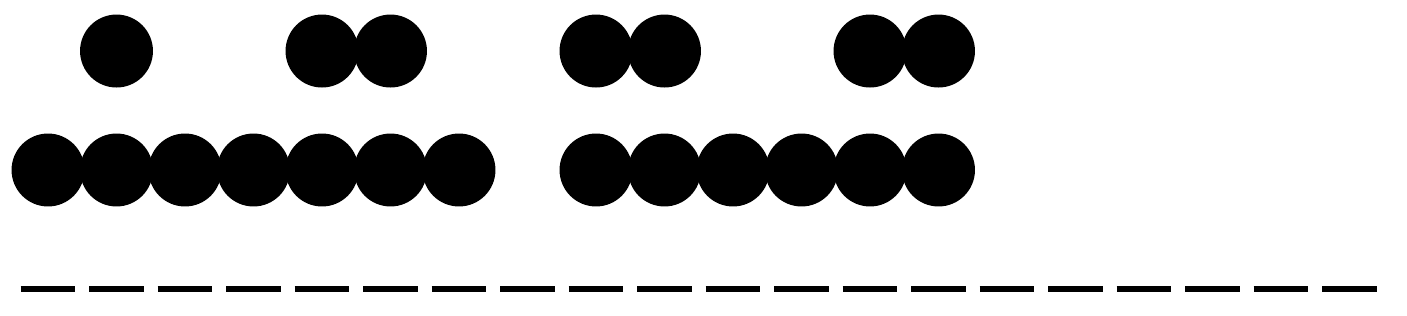}\qquad
\includegraphics[trim=24 110 24 160, clip, width=4cm]{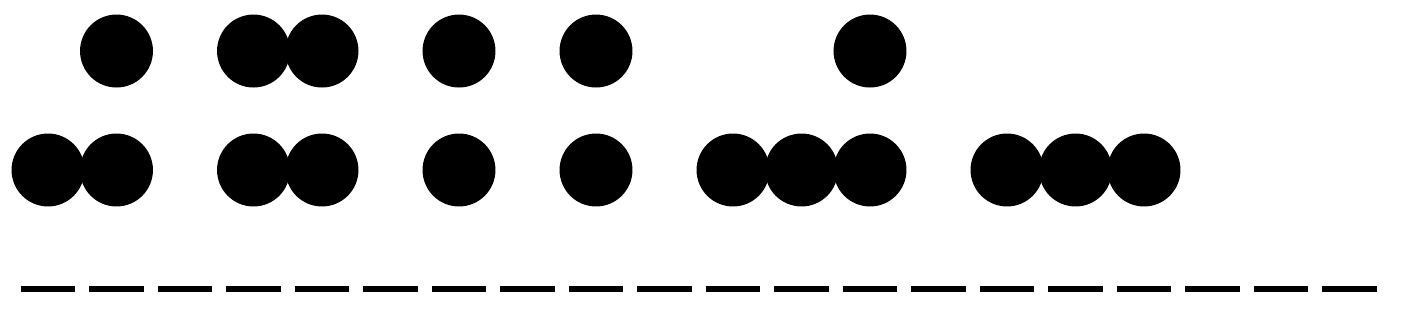}\qquad
\includegraphics[trim=24 110 24 160, clip, width=4cm]{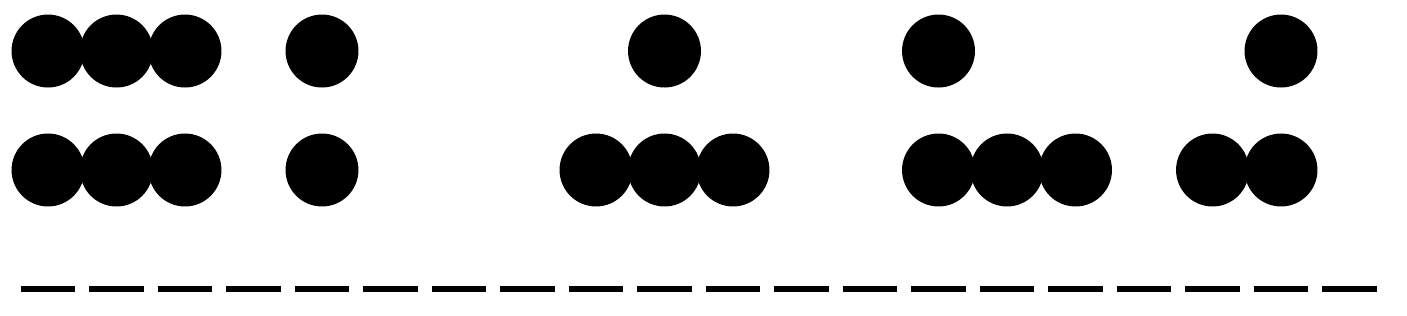}

\begin{minipage}{1.4cm}$t=30$\\[6mm]
\end{minipage}
\includegraphics[trim=24 110 24 160, clip, width=4cm]{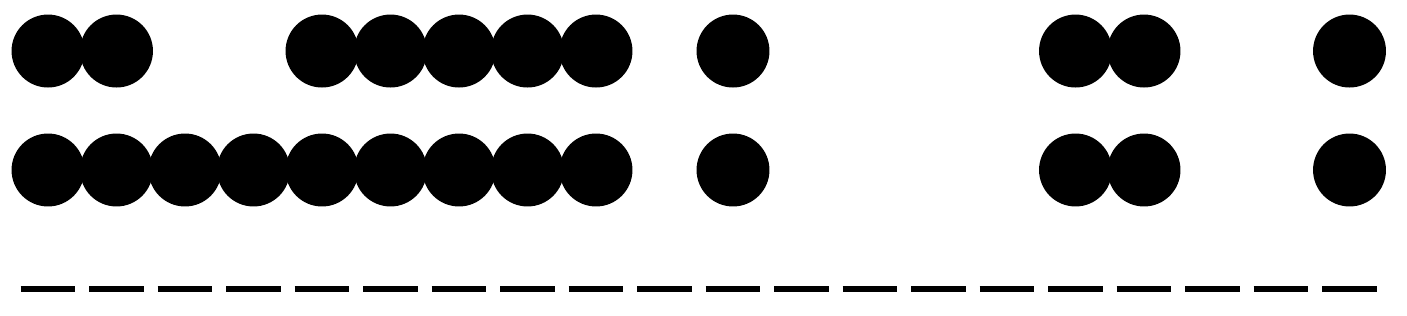}\qquad
\includegraphics[trim=24 110 24 160, clip, width=4cm]{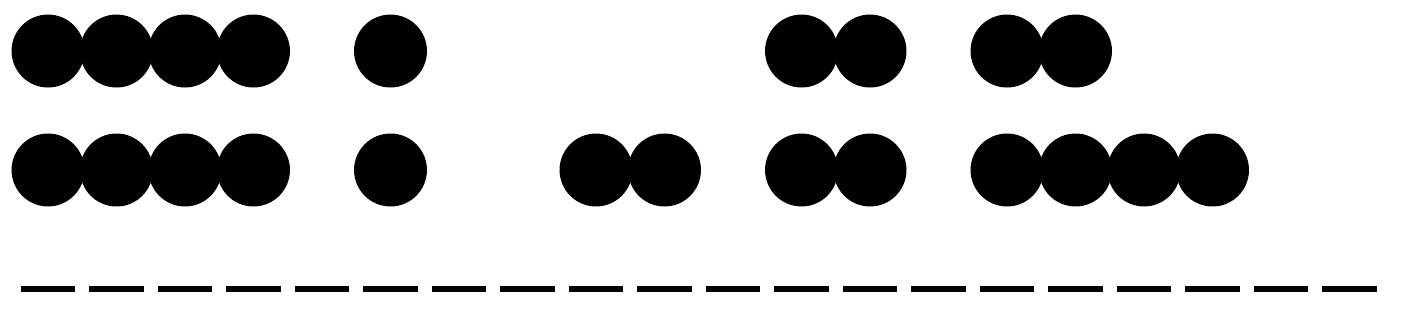}\qquad
\includegraphics[trim=24 110 24 160, clip, width=4cm]{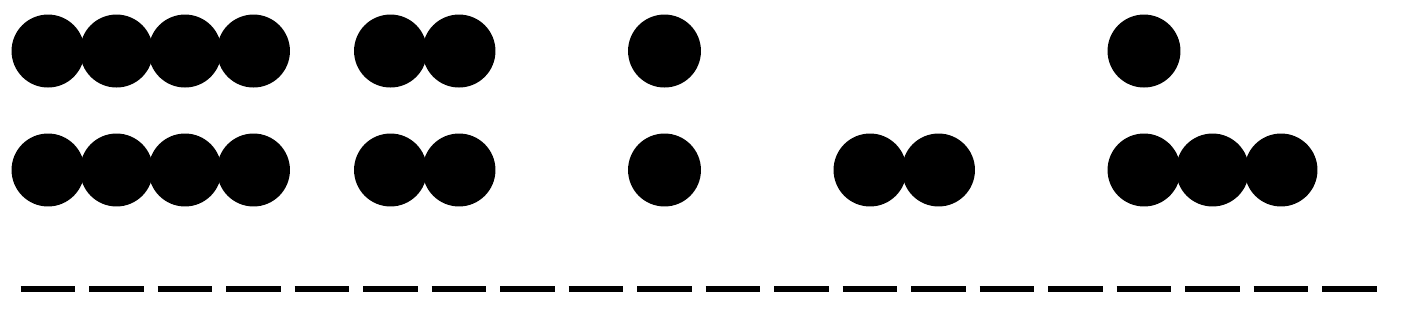}

\begin{minipage}{1.4cm}$t=40$\\[6mm]
\end{minipage}
\includegraphics[trim=24 110 24 160, clip, width=4cm]{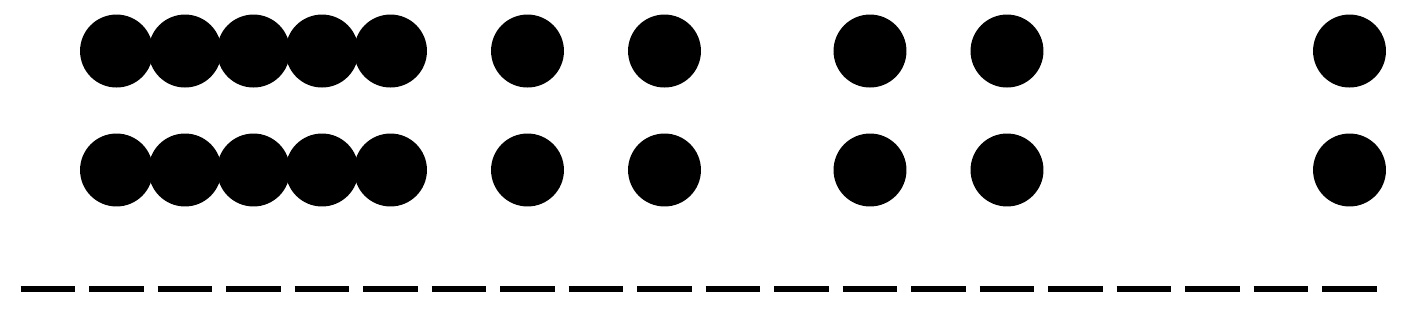}\qquad
\includegraphics[trim=24 110 24 160, clip, width=4cm]{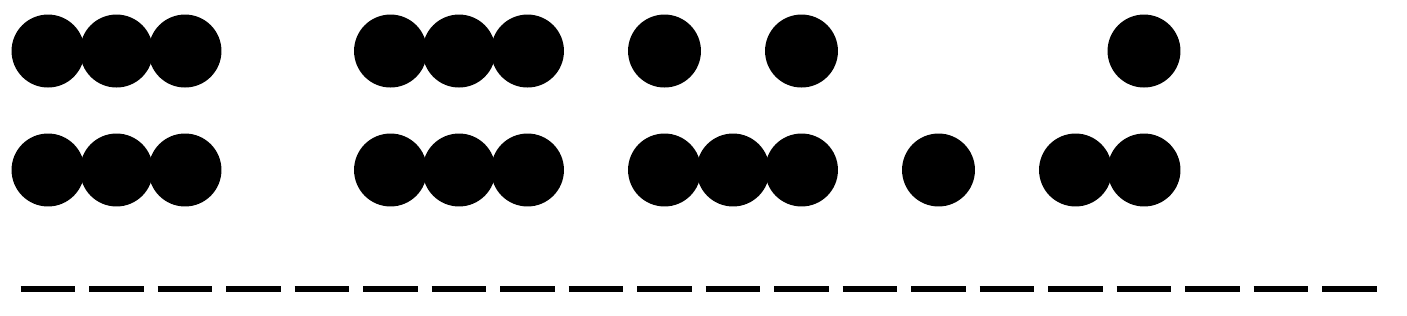}\qquad
\includegraphics[trim=24 110 24 160, clip, width=4cm]{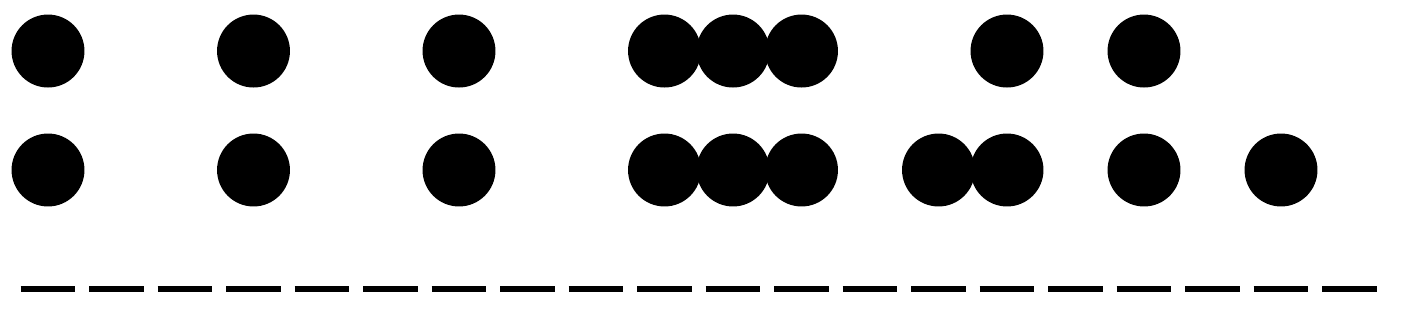}

\begin{minipage}{1.4cm}$t=50$\\[6mm]
\end{minipage}
\includegraphics[trim=24 110 24 160, clip, width=4cm]{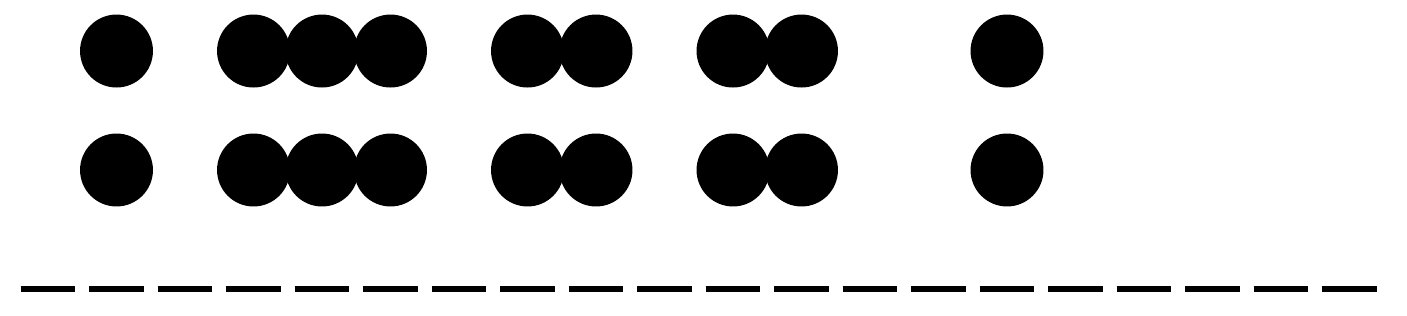}\qquad
\includegraphics[trim=24 110 24 160, clip, width=4cm]{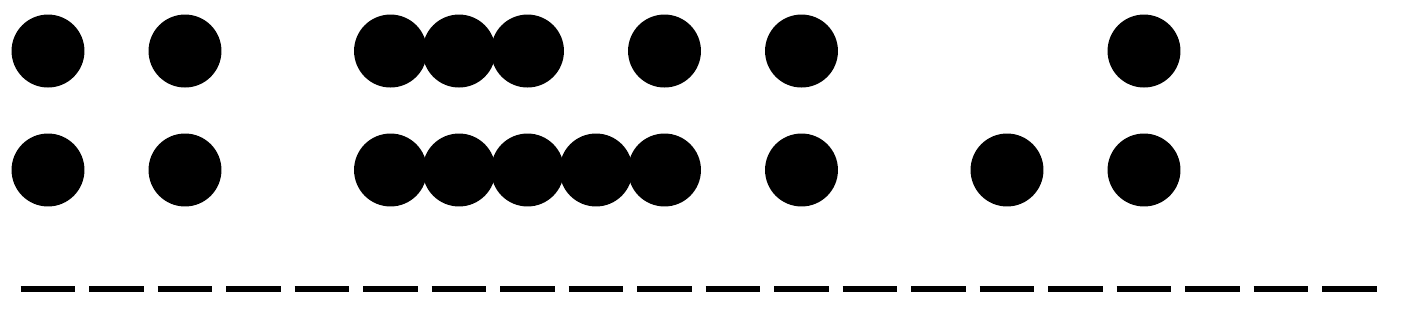}\qquad
\includegraphics[trim=24 110 24 160, clip, width=4cm]{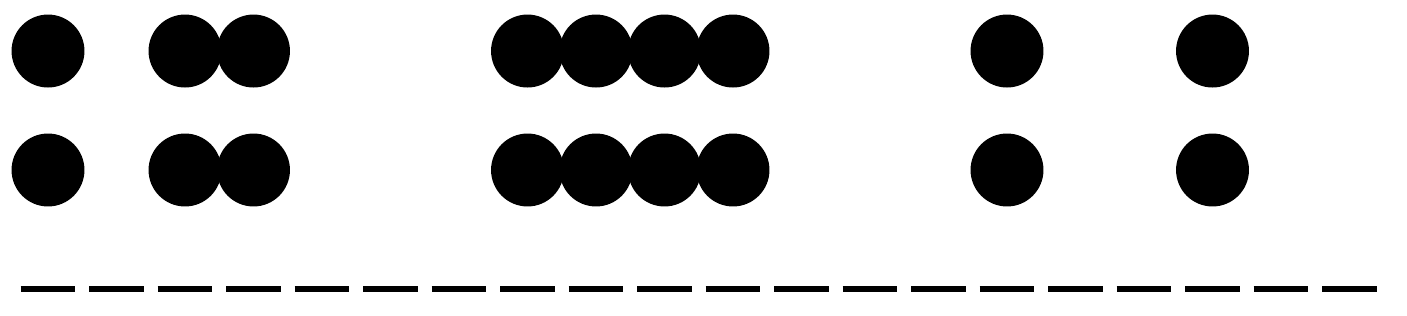}

\begin{minipage}{1.4cm}$t=60$\\[6mm]
\end{minipage}
\includegraphics[trim=24 110 24 160, clip, width=4cm]{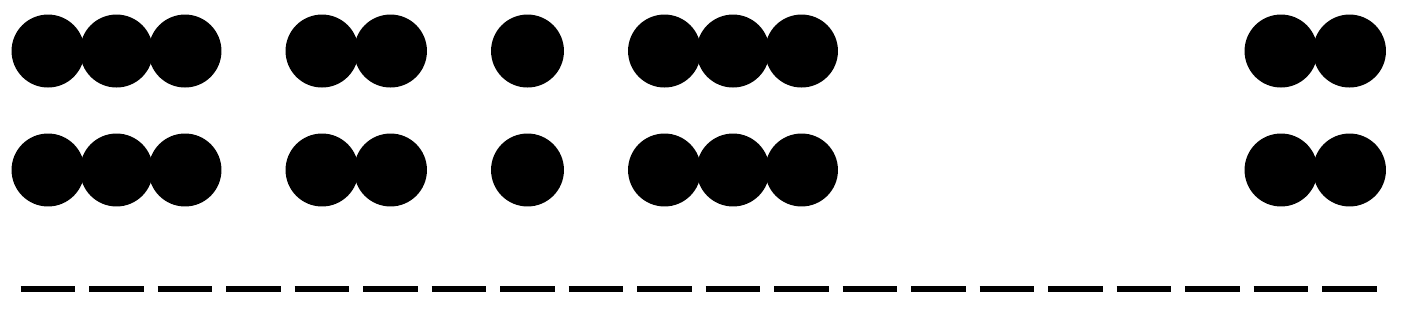}\qquad
\includegraphics[trim=24 110 24 160, clip, width=4cm]{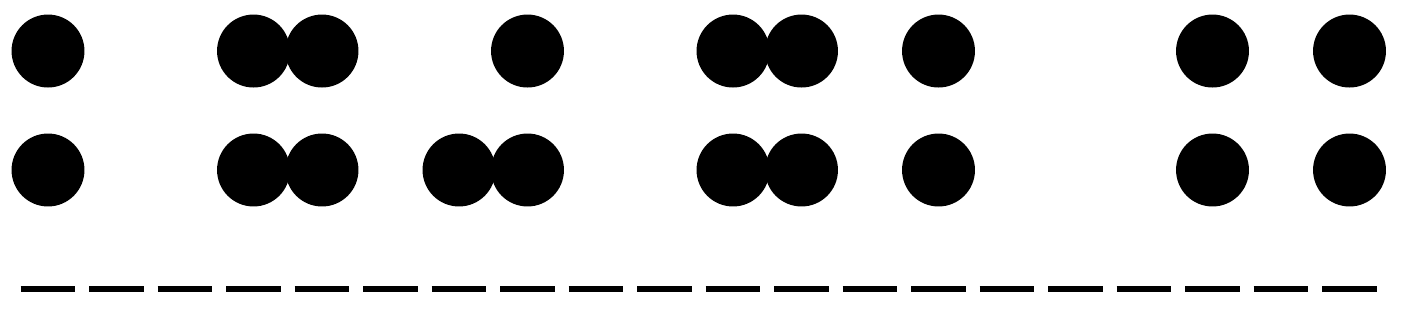}\qquad
\includegraphics[trim=24 110 24 160, clip, width=4cm]{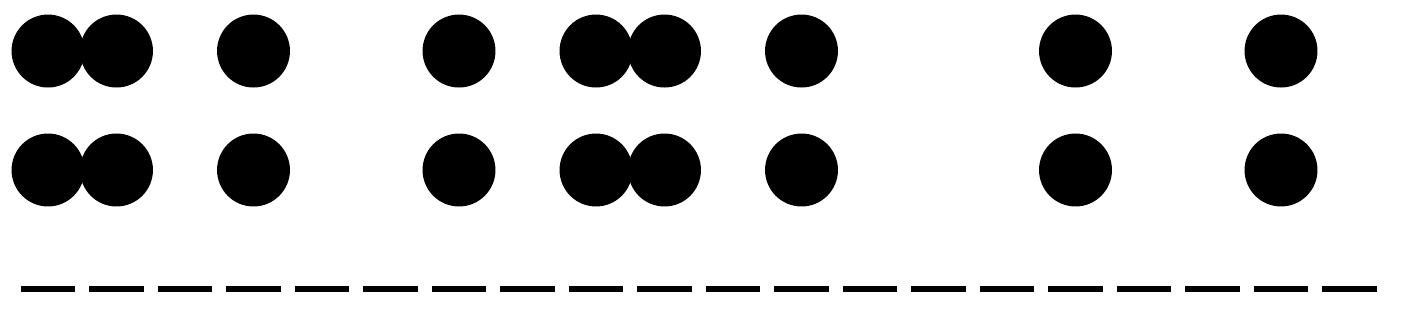}

\begin{minipage}{1.4cm}$t=70$\\[6mm]
\end{minipage}
\includegraphics[trim=24 110 24 160, clip, width=4cm]{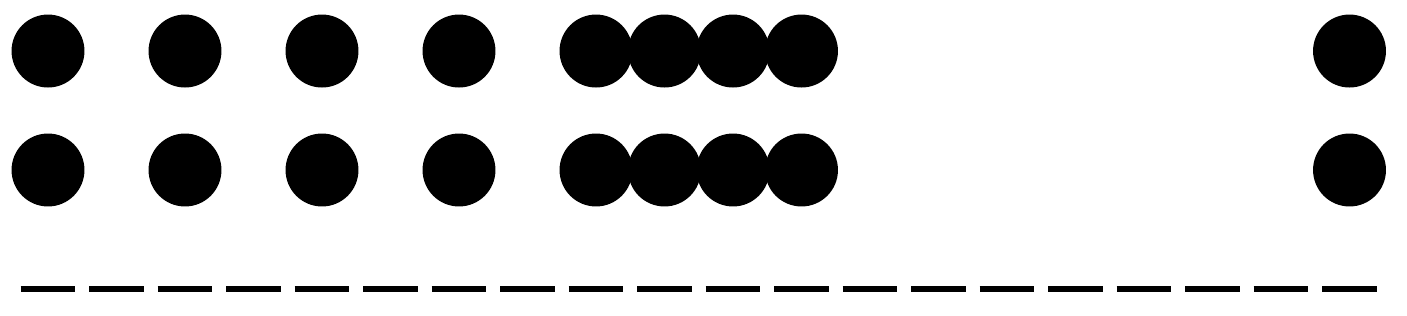}\qquad
\includegraphics[trim=24 110 24 160, clip, width=4cm]{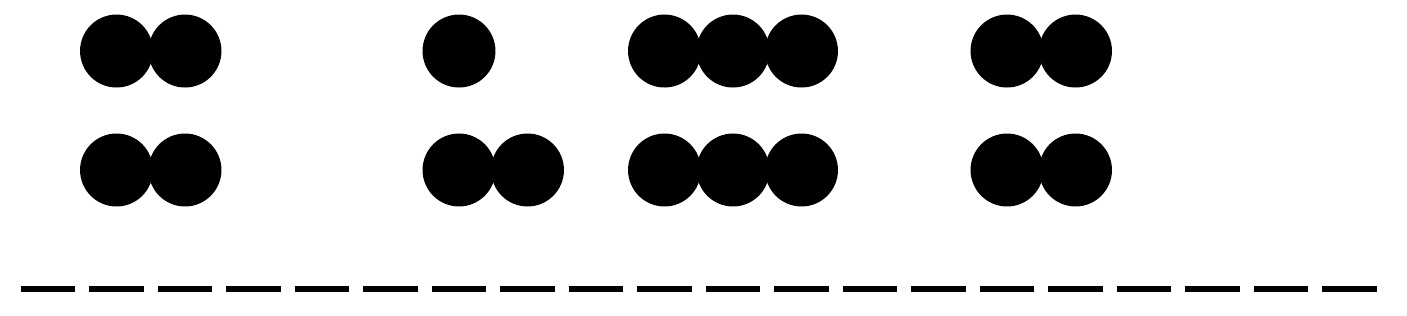}\qquad
\includegraphics[trim=24 110 24 160, clip, width=4cm]{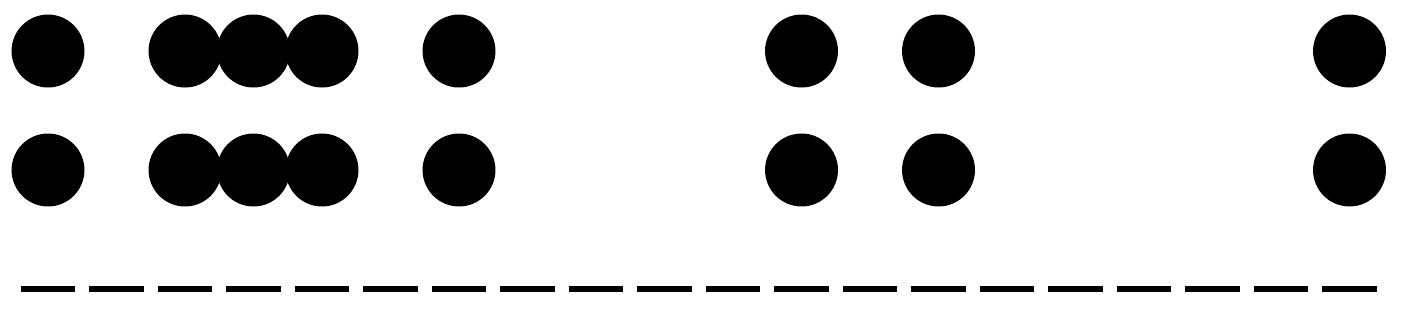}

\begin{minipage}{1.4cm}$t=80$\\[6mm]
\end{minipage}
\includegraphics[trim=24 110 24 160, clip, width=4cm]{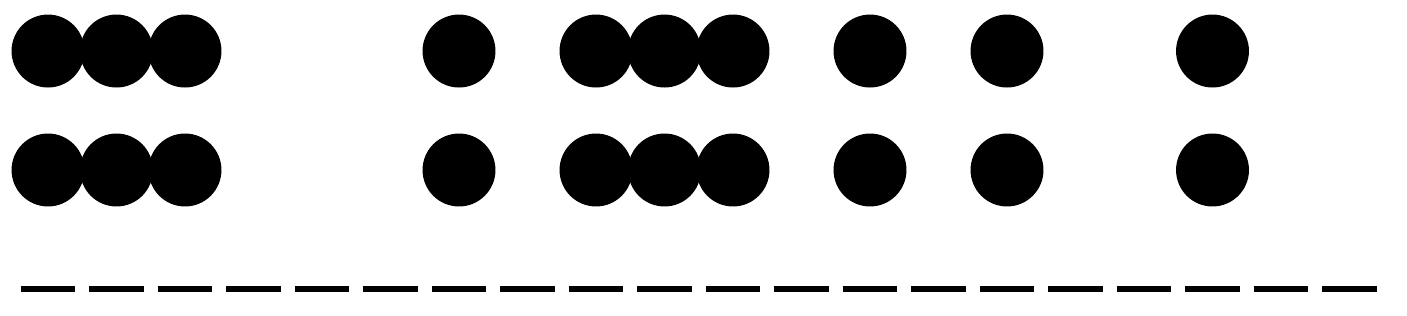}\qquad
\includegraphics[trim=24 110 24 160, clip, width=4cm]{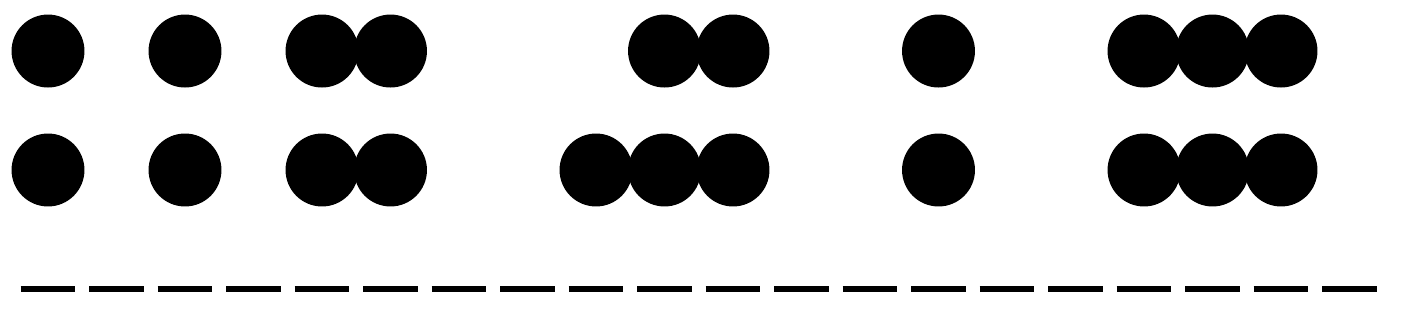}\qquad
\includegraphics[trim=24 110 24 160, clip, width=4cm]{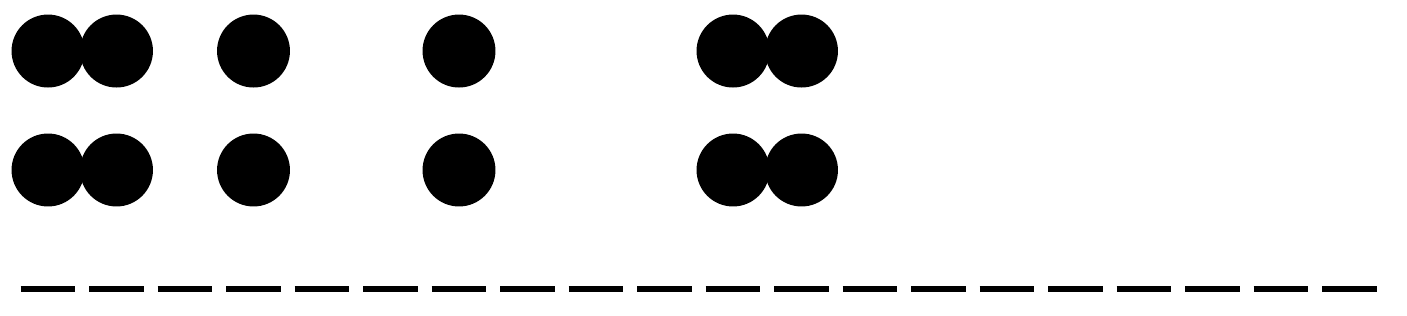}

\begin{minipage}{1.4cm}$t=90$\\[6mm]
\end{minipage}
\includegraphics[trim=24 110 24 160, clip, width=4cm]{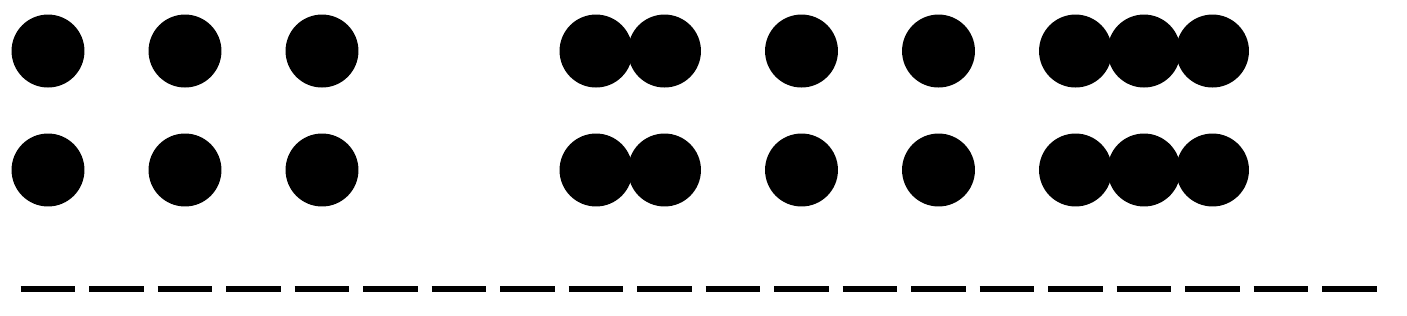}\qquad
\includegraphics[trim=24 110 24 160, clip, width=4cm]{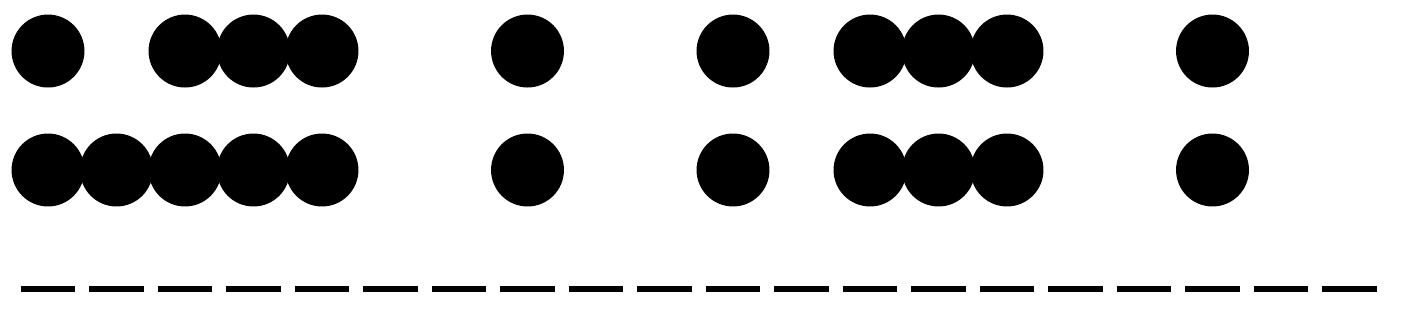}\qquad
\includegraphics[trim=24 110 24 160, clip, width=4cm]{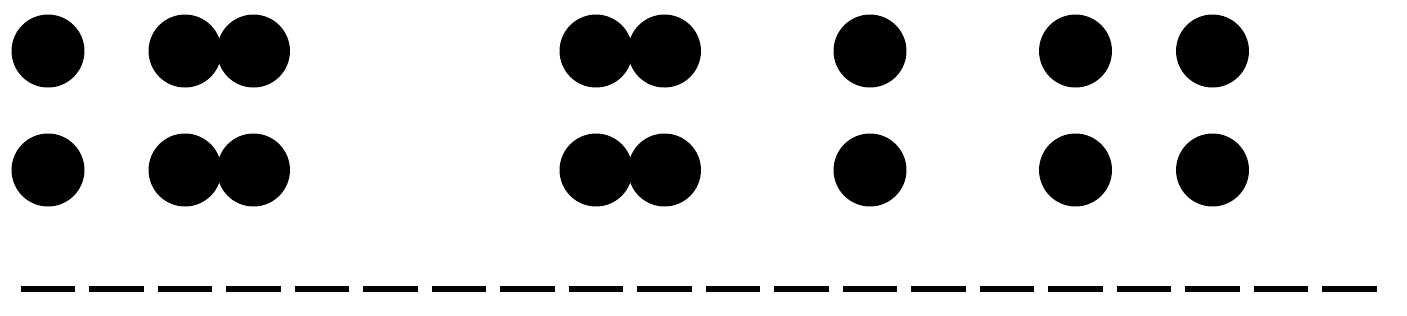}

\begin{minipage}{1.4cm}$t=100$\\[6mm]
\end{minipage}
\includegraphics[trim=24 110 24 160, clip, width=4cm]{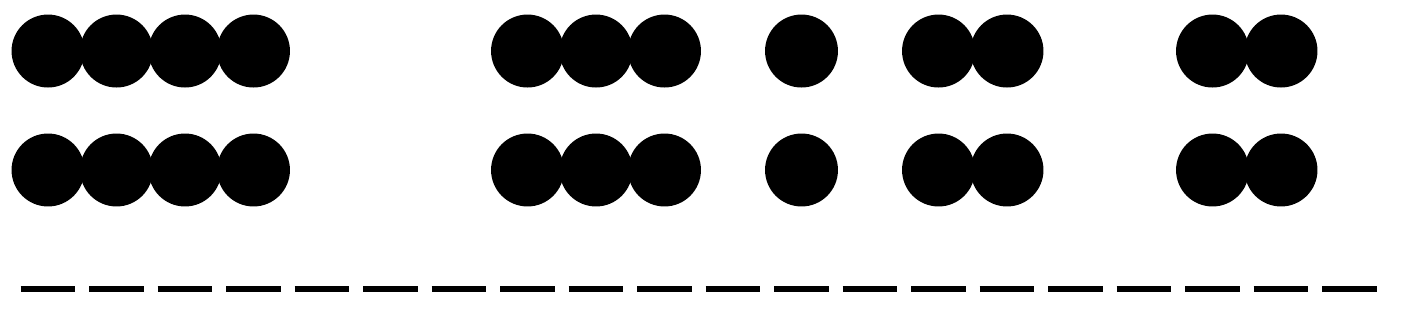}\qquad
\includegraphics[trim=24 110 24 160, clip, width=4cm]{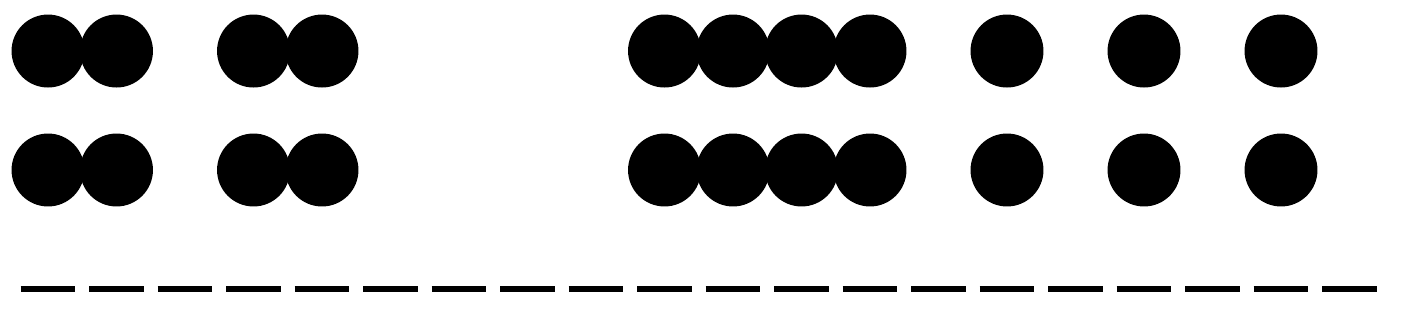}\qquad
\includegraphics[trim=24 110 24 160, clip, width=4cm]{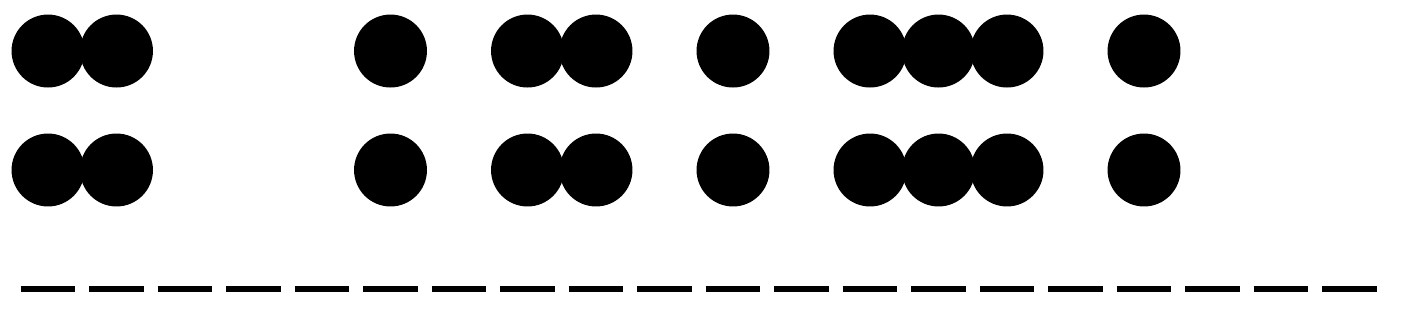}
\caption{Six trajectories of TASEP with three different random jump time sequences (left to right) and two different initial conditions (depicted on top of each other for each $t$).}
\label{fig:sim}\end{figure}

One clearly sees random attraction to a single trajectory at $t=40$ in the left column, at $t=100$ in the middle and at $t=50$ in the right column. These numbers are quite small when compared to the expected number of jumps that are needed to \TK{stumble upon} the particular jump sequence (of length 210) that was used to prove Lemma~\ref{lemma:fintime}. This number is an integer with 278 digits.

\TK{ Our second set of simulations aims at studying the  dependence of the synchronization time on the length of the chain. Here we consider the case that all interior hopping rates are equal, say $h_k=1$ for $1\leq k\leq n-1$. It is well known that the qualitative behavior of the dynamics depends significantly on the values of the entry rate $\alpha$ and of the exit rate $\beta$ (see e.g.,~\cite[Subsection 2.2.1]{solvers_guide}). In particular, one distinguishes between the phases:
\[
\begin{cases}
			\min \{ \alpha, \beta \} > \frac{1}{2}  	&	\text{ maximal-current (MC),} \\
		 \min \{ \alpha, \beta \} < \frac{1}{2},\; \alpha > \beta &\text{ high-density (HD)}, \\
		 \min \{ \alpha, \beta \} < \frac{1}{2},\; \alpha<  \beta &\text{ low-density (LD)}. 
\end{cases}
\]
}

\TK{
In the recent work~\cite{Gantert_etal_2020}, asymptotic results for mixing times are derived for simple exclusion processes with open boundaries that include the model described in the previous paragraph. More precisely, for $\varepsilon >  0$ the $\varepsilon$-mixing time $t_{\mbox{\tiny mix}}^n(\varepsilon)$ is the infimal time for which the total-variance distance between the stationary measure and the evolving probability distribution on the $2^n$ states of the chain of length $n$ is less than $\varepsilon$ no matter which initial state is chosen. Theorem~1.5 in~\cite{Gantert_etal_2020} states that in both the low-density and the high-density phases there exists a constant $C(\alpha, \beta)$ such that for all $0 < \varepsilon < 1$ one has
\begin{equation}\label{eqn:gantert}
1 \leq \liminf_{n \to \infty} \frac{t_{\mbox{\tiny mix}}^n(\varepsilon)}{n} \leq \limsup_{n \to \infty} \frac{t_{\mbox{\tiny mix}}^n(\varepsilon)}{n} \leq C(\alpha, \beta).
\end{equation}
As stated in~\cite[Theorem~1.3]{Gantert_etal_2020} this result extends to the case that either the entry or the exit is blocked, i.e., that either $\alpha=0$ or $\beta=0$. Note that the upper bound in the above inequality does not depend on $\varepsilon \in (0, 1)$. This is called the pre-cutoff phenomenon, see e.g.,~\cite[Chapter 18]{LevP17}. Observe, furthermore, that the \emph{coexistence line} $0 < \alpha = \beta < \frac{1}{2}$ where the LD and the HD phases meet is not included in either phase. This part of the phase space displays particularly interesting behavior and it is posed as an open question in~\cite{Gantert_etal_2020} what the large $n$ behavior of the mixing times is on this line segment.
Moreover, Conjecture~1.8 of~\cite{Gantert_etal_2020} states that in the maximal-current phase and on its boundary, i.e.~$\min \{ \alpha, \beta \} \geq \frac{1}{2}$, the mixing time is of order $n^{3/2}$ as opposed to order $n$ in the LD/HD phases. This conjecture is complemented in~\cite{Gantert_etal_2020} by an upper bound $t_{\mbox{\tiny mix}}^n(\varepsilon) \leq C n^3 \log(n)$ for the triple point $\alpha = \beta = \frac{1}{2}$ where, again, the constant $C$ does not depend on $\varepsilon$.
}

\TK{
One useful tool to obtain upper bounds for mixing times is the following connection to synchronization times: Denote by $\tau^{(n)}$ the random variable that equals the infimal time when the trajectories starting from the empty chain and from the fully occupied chain, both of length $n$, synchronize. If its tail probability satisfies $\mathbb P(\tau^{(n)} \geq s) \leq \varepsilon$ for some $s \geq 0$ then the $\varepsilon$-mixing time is bounded by $t_{\mbox{\tiny mix}}^n(\varepsilon) \leq s$ (see e.g.,~\cite[Lemma~2.2]{Gantert_etal_2020}).
}

\TK{
We now present numerical simulations that provide evidence in support of Conjecture~1.8 in~\cite{Gantert_etal_2020}, at least for most parts of the maximal-current phase, and that shed some light on the $n$-dependence of the mixing time for points on the coexistence line $0 < \alpha = \beta < \frac{1}{2}$. To this end we have computed for 35 different values of $(\alpha, \beta)$ the expectation of the synchronization time $\tau^{(n)}$ using Monte-Carlo simulations. In order to study the dependence on the length of the chain we have picked $40$ different values of chain lengths $n$ ranging from $n=11$ to $n=160$ that were used for all choices of $(\alpha, \beta)$. For each value of $n$, $\alpha$, and $\beta$ we have approximated the expectation $\mathbb E(\tau^{(n)})$ using 7,200 runs. As we explain in more detail below we found for all values of $(\alpha, \beta)$ that $\mathbb E(\tau^{(n)})$ is well described by a power law $C n^{\gamma}$. Using Markov's inequality $\mathbb P(\tau^{(n)} \geq s) \leq \frac{1}{s} \mathbb E(\tau^{(n)})$ one immediately obtains a bound on the mixing times, c.f.~\cite[Corollary~5.5]{LevP17}
\begin{equation}\label{eqn:mixing_bound}
t_{\mbox{\tiny mix}}^n(\varepsilon) \leq \frac{1}{\varepsilon} \mathbb E(\tau^{(n)}) \sim \frac{C}{\varepsilon} n^{\gamma}.
\end{equation}
Clearly, Markov's inequality is too crude to analyze   the phenomenon of cutoff or pre-cutoff. Nevertheless, it appears reasonable to extract the  exponents $\gamma$ of the power laws for the expectations and we do so by linear regression of $\log \mathbb E(\tau^{(n)})$ as a function of $\log n$. Our numerical results are depicted in~Figure~\ref{fig:exponents}. As one can see there we have only chosen points in the $(\alpha, \beta)$-plane that lie on or below the diagonal $\alpha = \beta$. The reason for this is that one may interchange the roles of $\alpha$ and $\beta$ by the particle-hole duality which leaves the expectations $\mathbb E(\tau^{(n)})$ invariant.
}
\begin{figure}
        \centering
				   \includegraphics[scale=0.9,trim=4cm 9cm 3cm 10cm,clip]{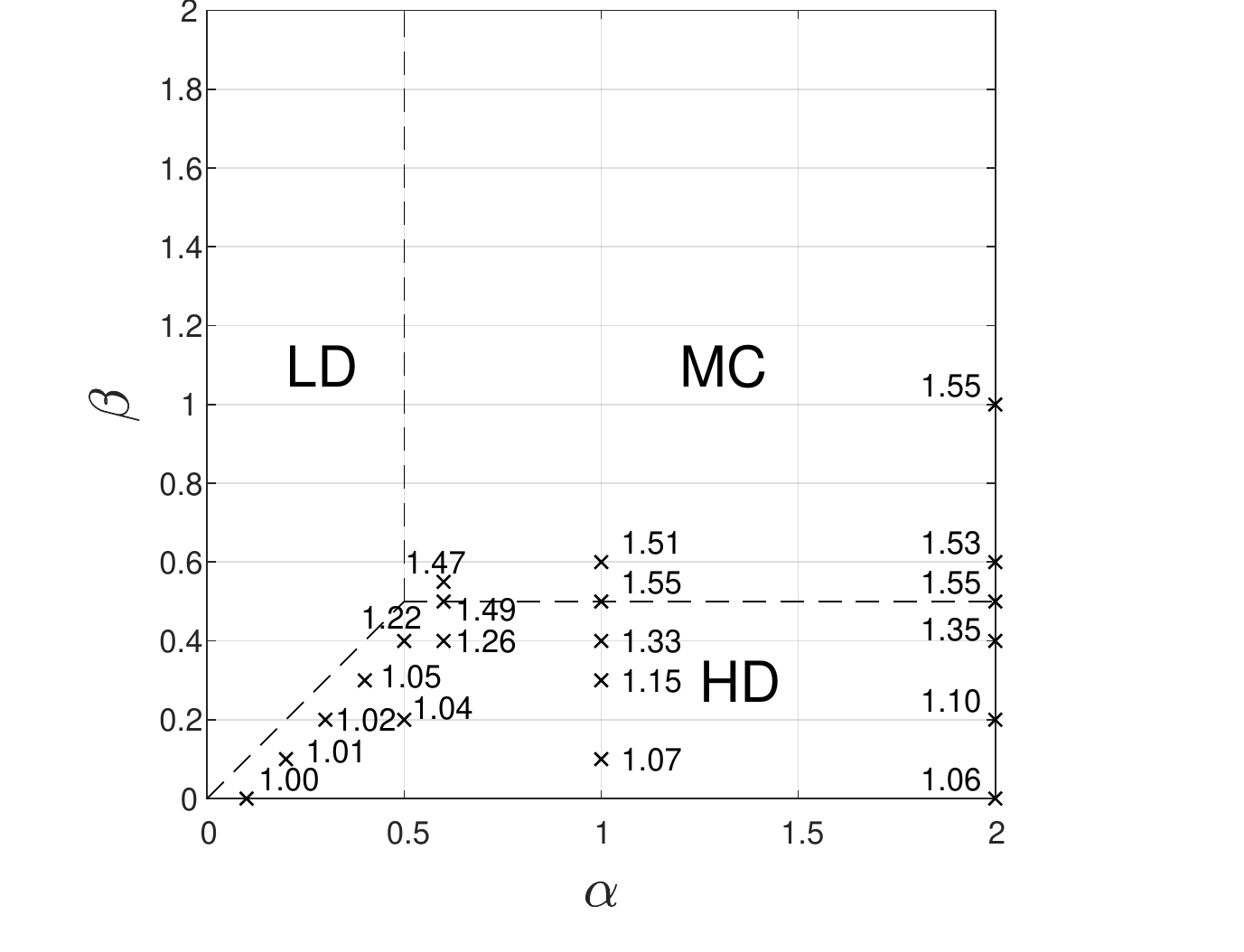}
        \\ \vspace*{1cm}
      				   \includegraphics[scale=0.9,trim=4cm 9cm 3cm 10cm,clip]{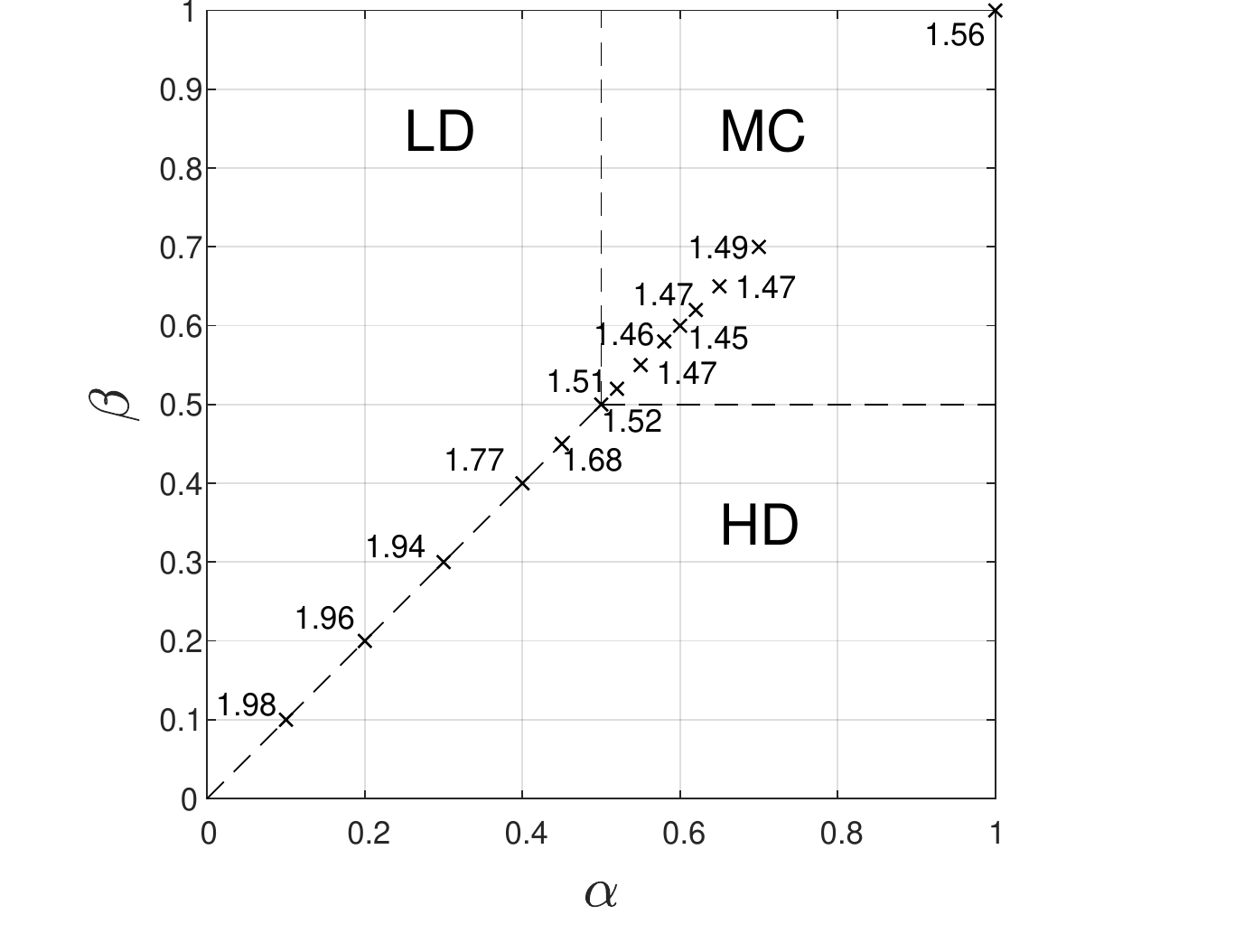}
   \caption{\TK{Numerical approximations for the exponent~$\gamma$ in the power
 law~$\mathbb E(\tau^{(n)}) \sim C n^{\gamma}$ for the expectation of the synchronization time~$\tau^{(n)}$. These were obtained for 35 points in the~$(\alpha, \beta)$-plane by Monte-Carlo simulations with 7,200 runs for each of the 11 chain lengths $n=80+8k$ with $k \in \{0,1,\ldots,10\}$.
Top: $\alpha>\beta$; Bottom $\alpha=\beta$. }}
\label{fig:exponents}
\end{figure} 

\TK{
Let us first look at the high-density phase in~Figure~\ref{fig:exponents} ($0 < \beta < \min \{\alpha, \frac{1}{2} \}$). The exponents are all $\geq 1$. Hence, the 
 bound~\eqref{eqn:mixing_bound} is consistent with the result \eqref{eqn:gantert} of~\cite{Gantert_etal_2020}. From the displayed exponents one is led to think that the
 bound~\eqref{eqn:mixing_bound} on the mixing times worsens as one approaches the maximal-current phase MC. This, however, might be a transient effect that disappears when $n$ tends to infinity. In fact, we generally observed that the exponents are mildly decreasing with $n$, with the notable exception of points on the coexistence line $0 < \alpha = \beta < \frac{1}{2}$. The dependence of the exponent~$\gamma$ on the chain length~$n$ is more pronounced in the high-density phase as one moves closer to the maximal-current phase. In order to demonstrate this phenomenon we have computed the exponents by linear regression for three different regions of chain lengths, $11 \leq n \leq 160$,  $40 \leq n \leq 160$, and $80 \leq n \leq 160$. The corresponding numbers are recorded in the tables in the Appendix. 
The numbers presented in~Figure~\ref{fig:exponents} are taken from the last segment $80 \leq n \leq 160$. 
}

\TK{
The dependence of the exponent on the length of the chain makes one wonder whether the expectations $\mathbb E(\tau^{(n)})$ are well described by a power law at all. We believe that this is the case. One indication for this is that in all of our experiments the deviation of $\mathbb E(\tau^{(n)})$ from $C n^{\gamma}$ did not exceed $1.6\%$ in the segment $80 \leq n \leq 160$. We refer the interested reader to the Appendix for a more detailed presentation of our numerical simulations where we also discuss the error that is introduced by the Monte-Carlo approximation of the expectation.
}
\subsection*{\TK{Summary of our findings}} 

\TK{
Conjecture~1.8~of~\cite{Gantert_etal_2020} which states that in the maximal current-phase mixing times are of order $n^{3/2}$ as $n$ tends to infinity is compatible with our simulations, except possibly for a region near the triple point $\alpha = \beta = \frac{1}{2}$. For example, at $\alpha = \beta = 0.6$
we obtain $\gamma = 1.455$. There are two reasons why we believe that this is not an artefact of the Monte-Carlo approximation. Firstly, we ran in addition a simulation with 28,800 runs for $\alpha = \beta = 0.6$ and obtained $\gamma = 1.458$. Secondly, Figure~\ref{fig:exponents} shows that the exponent falls slightly but consistently below $\frac{3}{2}$ for values of  $\alpha$  and $\beta$ that are close to $0.6$. Of course, this is by no means a contradiction to the conjecture in~\cite{Gantert_etal_2020}, as it might be a transient phenomenon that disappears for longer chains. 
}

\TK{
In regard to the behavior along the coexistence line that separates the high-density and the low-density phases ($0 < \alpha = \beta < \frac{1}{2}$),   addressed as Question~1.6 in~\cite{Gantert_etal_2020}, we see that the exponents of the power laws for the expectations $\mathbb E(\tau^{(n)})$ take values significantly above $\frac{3}{2}$,
 but remain below~$2$ and appear to be monotonically decreasing with~$\alpha = \beta$.
}


\section{Extensions}\label{sec:extensions}
The reasoning used in this paper can be extended or adapted to various other models. One example where this is possible
is the asymmetric simple exclusion process~(ASEP). 
ASEP follows the same rules as~TASEP except that the particles can hop in
both directions, see, e.g.,~\cite{TraW08}. One can model ASEP as an RDS either via two Poisson processes for each site, one for jump attempts to the left and one to the right, or via an additional sequence of random variables $d_i(\omega)\in\{-1,1\}$, where $-1$ denotes a hop to the left and $1$ to the right. In both cases, we need to add a Poisson process for modelling particles that enter the chain from the right.

The decisive consequence this extension of the model has on our formalism is that the jump order sequences $(k_i)_{i\in\Z}$ now take values in $\{-n-1,\ldots,-1,0,1,\ldots,n\}$. Here, $k=1,\ldots,n$ represents particles hopping from site $k$ to the right, $k=-1,\ldots,-n$ represents particles hopping from site $-k$ to the left and $k=0$ and $k=-n-1$ represent particles entering from the left or from the right, respectively, into the chain. As before, on a finite interval each finite jump order sequence occurs with positive probability that only depends on the length of the interval. Hence, Lemma \ref{lemma:fintime} remains valid (although the probability $\delta$ for the sequence constructed in the proof becomes smaller). Since all subsequent results in this paper build solely upon Lemma~\ref{lemma:fintime}, they thus remain valid if TASEP is replaced by ASEP.

Another modification would be to consider TASEP with time-periodic and, say, continuous transition rates. Then the construction of the RDS needs to be modified as the Poisson processes are not homogeneous any more and have no representation as partial sums of exponentially distributed variables. For such a construction the method described in \cite[Section 2.5]{Kingman} appears to be useful. The proof of the central Lemma~\ref{lemma:fintime}, however, can be easily adapted, as it only requires that there is a positive lower bound on each transition rate. Even in the case that transition rates are temporarily zero (but not identically equal to zero) one may prove Lemma~\ref{lemma:fintime} in the periodic case for times $t$ large enough (\LG{which would still be enough to satisfy property (1) of Theorem \ref{thm:singleton}}, cf.~Remark~\ref{rem:fintime}) as one might need to wait for several periods to realize the jump sequence prescribed in the proof of Lemma~\ref{lemma:fintime}.

\section{Conclusion} 
In this paper we have shown that random attraction to a single trajectory occurs almost surely in the TASEP model with finite chain length. This implies the existence of a random attractor that almost surely consists of single trajectories. 

For the many dynamical systems and processes that were modeled using TASEP, ranging from mRNA translation to vehicular traffic,
 this implies that in the long run these systems are insensitive to the initial conditions. Alternatively, this may be interpreted as saying that a perturbation of the state of the system   is ``filtered out''.

In order to rigorously prove our main result 
  we first  reformulated TASEP as a random dynamical system (RDS), which is a contribution of independent interest. Our construction of an RDS relies on the definition of jump times generated via generalized Poisson processes that are attached to the sites of the chain. This modelling is consistent with the fact that the exponential transition rates in TASEP are usually chosen site dependent. We have, moreover, shown 
  \TK{that random attraction to a single trajectory is a stronger property than contractiveness for the corresponding evolution of probability distributions. Rather, it is a consequence of the fact
  }
 that there exists a sequence of clock ticks, with non zero probability, that transfers the system to the same final state from \emph{any} initial state.

\section{Appendix} 

\TK{
As  described in Section \ref{sec:NumSim} the goal of the second set of numerical simulations was to test the hypothesis that the expectation $\mathbb E(\tau^{(n)})$ of  the synchronization time $\tau^{(n)}$ between the initially empty and the fully occupied chain of length $n$ obeys asymptotically (in $n$) a power law of the form $C n^{\gamma}$ and to determine the value of the exponent. As conjectured in~\cite{Gantert_etal_2020} for the corresponding mixing times, the value of the exponent should depend on the values of the entry rate $\alpha$ and the exit rate $\beta$. More precisely, it should depend on the location of $(\alpha, \beta)$ in the phase diagram, see Figure~\ref{fig:exponents}. 
}

\TK{
For the simulations we have used for each value of the parameters $(\alpha, \beta)$ 40 different chain lengths, varying from 11 to 20 in steps of 1, from 20 to 40 in steps of 2, from 40 to 80 in steps of 4, and from 80 to 160 in steps of 8. For each of these values we performed 7,200 runs to approximate the expectation $\mathbb E(\tau^{(n)})$ of the synchronization time. The values for the constant $C$ and for the exponent $\gamma$ of the power law were then obtained by linear regression of $\log \mathbb E(\tau^{(n)})$ as a function of $\log n$. Due to the available computer power it was feasible to consider chains up to length 160 for which the number of jump attempts is of the order of 1,000,000 before synchronization occurs. It turned out that for some choices of $(\alpha, \beta)$ in the phase plane chain length $n=160$ is too small to provide an asymptotic result. In order to get some idea how close we are to the asymptotic regime (assuming that this  exists) we have computed the exponent for three different regimes: $\gamma_0$ denotes the value obtained from the 11 values of our chosen chain lengths that lie in $[80,160]$,
$\gamma_1$ is the value obtained from the 21 values in $[40,160]$, and $\gamma_2$ denotes the value obtained from all 40 chain lengths. Looking at the numbers presented in Tables~\ref{table:numbers_offdiagonal} and~\ref{table:numbers_diagonal}, one sees that the exponents decrease with growing chain length with the notable exception of the coexistence line between HD and LD. Moreover, the differences between $\gamma_0$ and $\gamma_1$ are quite small except for the points in the high-density phase that are close to the maximal-current phase.
}

\TK{
In order to gain some insight on how good the Monte-Carlo approximations are, the 7,200 runs were split into 9 runs of 800 trials each and the empirical standard deviation of the resulting 9 values was computed. The such obtained numbers are denoted by $\sigma_j$, where $j \in \{0,1,2\}$ indicates the segment of chain lengths used for the computation.
}

\TK{
 The numerical results we have described so far provide  
 little information on 
 the quality of the approximation of the 
expectations~$\mathbb E(\tau^{(n)})$ by the power law. As a measure to judge this quality we have computed $\Delta_j$ as the maximum of the relative deviations of  $\mathbb E(\tau^{(n)})$ from $C_j n^{\gamma_j}$,
 where the maximum is taken over all values of $n$ that lie in the relevant segments, i.e., in $[80,160]$ for $j=0$, in $[40,160]$ for $j=1$, and in $[11,160]$ for $j=2$. From these numbers one may deduce that the power law is indeed a good approximation. The maximal relative error $\Delta_0$ lies between $0.2\%$ and $1.6\%$ for all 35 values of $(\alpha, \beta)$. We consider the corresponding exponents $\gamma_0$ as our best justified approximations to the asymptotic value of the exponent (if it exists) and this is why they are presented in Figure~\ref{fig:exponents}. In addition, we list the values of the corresponding constants $C_0$ in Tables~\ref{table:numbers_offdiagonal} and~\ref{table:numbers_diagonal}. Note that only the maximal relative errors $\Delta_2$ may become unsatisfactorily large, going up to $19\%$. However, all larger values of $\Delta_2$ may be explained by the discrepancy between $\gamma_0$ and $\gamma_2$. In these cases one would need to consider longer chains in order to test the power law and to obtain realistic values for the exponents.
}

\begin{table}[htp] 
\begin{tabular}{ | p{0.6cm}|p{0.6cm}||p{2.3cm}|p{0.8cm}|p{0.8cm}||p{2.3cm}|p{0.8cm}||p{2.3cm}|p{0.8cm}|  }
\hline
\multicolumn{9}{|c|}{HD: $0\leq\beta < 1/2$ and~$\beta<\alpha$} \\
\hline
\multirow{2}{*}{$\alpha$}& 
  \multirow{2}{*} {$\beta$}&
	\multicolumn{3}{c||}{ $80 \leq n\leq 160$ }     &
\multicolumn{2}{c||}{ $40 \leq n\leq 160$ }  &
\multicolumn{2}{c|}{ $11 \leq n\leq 160$ }   \\
\cline{3-9}
& &  
 \Centering $\gamma_0\pm \sigma_0$ &
\Centering $\Delta_0$ &
\Centering $C_0$&
\Centering $\gamma_1\pm \sigma_1$ & 
\Centering$\Delta_1$ & 
\Centering$\gamma_2\pm \sigma_2$ & 
\Centering$\Delta_2$ \\
\hline
$0.1$ & $0.0$ & $1.001 \pm 0.004 $ &$0.2\%$ & $11.1$ &$ 1.003 \pm 0.001$ & $ 0.4\%$ & $1.003\pm 0.001$ & $0.5\%$\\
\hline
$0.2$ & $0.1$ & $1.008 \pm 0.009 $ &$0.3\%$ & $10.9$ &$ 1.012 \pm 0.002$ & $ 0.6\%$ & $1.036\pm 0.002$ & $3.2\%$\\
\hline
$0.3$ & $0.2$ & $1.024 \pm 0.011 $ &$0.4\%$ & $10.4 $ &$ 1.034 \pm 0.005$ & $ 0.7\%$ & $1.097\pm 0.003$ & $12\%$\\
\hline
$0.4$ & $0.3$ & $1.045 \pm 0.013 $ &$0.3\%$ & $10.3$ &$ 1.078 \pm 0.005$ & $ 2.5\%$ & $1.226\pm 0.003$ & $19\%$\\
\hline
$0.5$ & $0.2$ & $1.043 \pm 0.008 $ &$0.4\%$ & $5.11$ &$ 1.053 \pm 0.003$ & $ 0.7\%$ & $1.091\pm 0.002$ & $6.1\%$\\
\hline
$0.5$ & $0.4$ & $1.222 \pm 0.012 $ &$0.8\%$ & $5.36$ &$ 1.306 \pm 0.005$ & $ 4.6\%$ & $1.441\pm 0.002$ & $12\%$\\
\hline
$0.6$ & $0.4$ & $1.260 \pm 0.014 $ &$0.7\%$ & $3.67$ &$ 1.316 \pm 0.008$ & $ 2.8\%$ & $1.417\pm 0.002$ & $8.9\%$\\
\hline
$1.0$ & $0.1$ & $1.068\pm 0.003 $ &$0.2\%$ & $3.25$ &$ 1.084 \pm 0.001$ & $ 0.9\%$ & $1.131\pm 0.001$ & $6.2\%$\\
\hline
$1.0$ & $0.3$ & $1.150\pm 0.006 $ &$0.5\%$ & $3.86$ &$ 1.183\pm 0.004$ & 
$ 2.8\%$ & $1.296\pm 0.001$ & $13\%$\\
\hline
$1.0$ & $0.4$ & $1.328\pm 0.017 $ &$0.6\%$ & $2.45$ &$ 1.393\pm 0.003$ & 
$ 3.2\%$ & $1.490\pm 0.002$ & $9.0\%$\\
\hline
$2.0$ & $0.0$ & $1.058\pm 0.002 $ &$0.2\%$ & $2.78$ &$ 1.075\pm 0.001$ & 
$ 0.9\%$ & $1.124\pm 0.001$ & $6.5\%$\\
\hline
$2.0$ & $0.2$ & $1.102\pm 0.006 $ &$0.3\%$ & $3.50$ &$ 1.127\pm 0.002$ & 
$ 1.4\%$ & $1.213\pm 0.002$ & $13\%$\\
\hline
$2.0$ & $0.4$ & $1.346\pm 0.014 $ &$0.9\%$ & $2.22$ &$ 1.409\pm 0.004$ & 
$ 3.5\%$ & $1.524\pm 0.002$ & $11\%$\\
\hline
\multicolumn{9}{|c|}{MC without $\{\alpha=\beta\}$ : $ \beta \geq  1/2$ and~$\beta<\alpha$} \\
\hline
$0.6$ & $0.5$ & $1.493\pm 0.014 $ &$0.6\%$ & $1.39$ &$ 1.494\pm 0.008$ & 
$ 0.8\%$ & $1.521 \pm 0.003$ & $3.2\%$\\
\hline
$0.6$ & $0.55$ & $1.469 \pm 0.017 $ &$0.8 \%$ & $1.44$ &$ 1.475 \pm 0.010$ & 
$ 1.1\%$ & $1.508 \pm 0.005$ & $3.7\%$\\
\hline
$1.0$ & $0.5$ & $1.553 \pm 0.021 $ &$0.9 \%$ & $0.93$ &$ 1.552 \pm 0.009$ & 
$ 1.0\%$ & $1.574 \pm 0.004$ & $2.3\%$\\
\hline
$1.0$ & $0.6$ & $1.506 \pm 0.028 $ &$0.8 \%$ & $1.02$ &$ 1.513 \pm 0.011$ & 
$ 0.9\%$ & $1.547 \pm 0.004$ & $3.6\%$\\
\hline
$2.0$ & $0.5$ & $1.550 \pm 0.022 $ &$0.9 \%$ & $0.94$ &$ 1.565 \pm 0.007$ & 
$ 1.2\%$ & $1.610 \pm 0.002$ & $5.6\%$\\
\hline
$2.0$ & $0.6$ & $1.526 \pm 0.032 $ &$0.8 \%$ & $0.91$ &$ 1.531 \pm 0.014$ & 
$ 1.2\%$ & $1.580 \pm 0.004$ & $6.0\%$\\
\hline
$2.0$ & $1.0$ & $1.554 \pm 0.021 $ &$1.0 \%$ & $0.72$ &$ 1.574 \pm 0.009$ & 
$ 1.1\%$ & $1.622 \pm 0.002$ & $5.1\%$\\
\hline   
\end{tabular} \vspace*{0.3cm}
\caption{\label{table:numbers_offdiagonal} \TK{The values $\gamma_j$, $\sigma_j$, $\Delta_j$ ($j=0,1,2$) and $C_0$ when $\beta < \alpha$.}}
\end{table}
  
	
\begin{table}[htp]
\begin{tabular}{  |p{0.7cm}||p{2.3cm}|p{1cm}|p{0.8cm}||p{2.3cm}|p{0.8cm}||p{2.3cm}|p{0.8cm}|  }
\hline
\multicolumn{8}{|c|}{ Coexistence line between HD and LD: 
$0<\alpha=\beta<1/2$} \\
\hline
 \hspace*{-.25cm}\multirow{2}{*}{ $\alpha=\beta$}& 
	\multicolumn{3}{c||}{ $80 \leq n\leq 160$ }       &
\multicolumn{2}{c||}{ $40 \leq n\leq 160$ }  &
\multicolumn{2}{c|}{ $11 \leq n\leq 160$ }  \\
\cline{2-8}
&  \Centering $\gamma_0\pm \sigma_0$ & \Centering$\Delta_0$ &\Centering $C_0$&\Centering
$\gamma_1\pm \sigma_1$ &
 \Centering$\Delta_1$ & 
\Centering $\gamma_2\pm \sigma_2$ &
\Centering $\Delta_2$ \\
\hline
$0.1$ &  $1.982 \pm 0.046 $ &$1.1\%$ & $2.46$ &$ 1.974 \pm 0.010$ & $ 1.1\%$ & $1.953\pm 0.004$ & $4.3\%$\\
\hline
$0.2$ &  $1.963 \pm 0.028 $ &$1.1\%$ & $1.16$ &$ 1.958 \pm 0.010$ & $ 1.6\%$ & $1.915\pm 0.004$ & $4.8\%$\\
\hline
$0.3$ &  $1.936 \pm 0.025 $ &$1.3\%$ & $0.70$ &$ 1.895 \pm 0.010$ & $ 2.0\%$ & $1.842\pm 0.003$ & $4.6\%$\\
\hline
$0.4$ &  $1.770\pm 0.014 $ &$1.4\%$ & $0.84$ &$ 1.754 \pm 0.008$ & $ 1.8\%$ & $1.725\pm 0.003$ & $2.7\%$\\
\hline
$0.45$ &  $1.676\pm 0.024 $ &$0.9\%$ & $0.96$ &$ 1.659 \pm 0.011$ & $ 1.4\%$ & $1.645\pm 0.003$ & $2.3\%$\\
\hline
\multicolumn{8}{|c|}{Within the region of MC:
$\{1/2\leq \alpha=\beta\}$ } \\
\hline
$0.5$ &  $1.524\pm 0.028 $ &$0.6\%$ & $1.47$ &$ 1.544 \pm 0.010$ & $ 1.1\%$ & $1.564\pm 0.004$ & $2.5\%$\\
\hline
$0.52$ &  $1.511\pm 0.016 $ &$0.8\%$ & $1.41$ &$ 1.514 \pm 0.003$ & $ 0.9\%$ & $1.540\pm 0.005$ & $2.6\%$\\
\hline
$0.55$ &  $1.468\pm 0.021 $ &$1.6\%$ & $1.54$ &$ 1.478 \pm 0.007$ & $ 1.7\%$ & $1.515\pm 0.005$ & $3.9\%$\\
\hline
$0.58$ &  $1.461\pm 0.024 $ &$0.7\%$ & $1.46$ &$ 1.471 \pm 0.008$ & $ 1.2\%$ & $1.502\pm 0.004$ & $3.9\%$\\
\hline
$0.6$ &  $1.455\pm 0.023 $ &$1.0\%$ & $1.44$ &$ 1.465 \pm 0.007$ & $ 1.5\%$ & $1.497\pm 0.003$ & $4.4\%$\\
\hline
$0.62$ &  $1.472\pm 0.014 $ &$0.9\%$ & $1.28$ &$ 1.471 \pm 0.006$ & $ 0.9\%$ & $1.499\pm 0.003$ & $3.2\%$\\
\hline
$0.65$ &  $1.472\pm 0.012 $ &$0.8\%$ & $1.23$ &$ 1.477 \pm 0.006$ & $ 0.9\%$ & $1.501\pm 0.003$ & $3.2\%$\\
\hline
$0.7$ &  $1.492\pm 0.020 $ &$0.9\%$ & $1.07$ &$ 1.495 \pm 0.008$ & $ 0.9\%$ & $1.513\pm 0.004$ & $3.2\%$\\
\hline
$1.0$ &  $1.556\pm 0.014 $ &$0.7\%$ & $0.73$ &$ 1.563 \pm 0.009$ & $ 1.0\%$ & $1.590\pm 0.003$ & $3.7\%$\\
\hline
$2.0$ &  $1.582\pm 0.014 $ &$1.0\%$ & $0.63$ &$ 1.592 \pm 0.007$ & $ 1.3\%$ & $1.657\pm 0.003$ & $7.5\%$\\
\hline
\end{tabular}
 \vspace*{0.3cm}
\caption{\label{table:numbers_diagonal} \TK{
The values  $\gamma_j$, $\sigma_j$, $\Delta_j$ ($j=0,1,2$) and $C_0$ when $  \alpha=\beta$.}}
\end{table}

\vspace{2mm}

{\bf Acknowledgment:} 
The authors thank  BayFOR, the University of Bayreuth, and Tel Aviv University
for their generous    support to   
Bilateral Workshops  of scholars from 
the University of Bayreuth  and Tel Aviv University.

We thank Melanie Birke and Walter Olbricht for making us aware of the literature on generalized Poisson processes. We also thank Peter Kloeden for discussing the construction of the TASEP-RDS with us.

\LG{Finally, we thank an anonymous reviewer who made us aware of the concepts of graphical representations, grand couplings, and the related literature\TK{, and who suggested to extend our numerical experiments on synchronization times to the $(\alpha, \beta)$-plane.}}

\end{document}